\documentclass[12pt,a4paper]{article}

\usepackage{latexsym}
\usepackage{amsmath}
\usepackage{amssymb}
\usepackage{amsthm}
\usepackage{amscd}
\usepackage{mathrsfs}
\usepackage[all]{xy}
\usepackage{graphicx}
\usepackage{comment}
\usepackage{arydshln}
\usepackage{mathtools}
\usepackage{colonequals}
\usepackage{braket}

\newtheorem{thm}{Theorem}[section]
\newtheorem{prop}[thm]{Proposition}
\newtheorem{defn}[thm]{Definition}
\newtheorem{lem}[thm]{Lemma}

\newtheorem{rem}[thm]{Remark}
\newtheorem{eg}[thm]{Example}
\newtheorem{notn}[thm]{Notation}

\makeatletter

\@addtoreset{equation}{section}
\makeatother

\makeatletter
\newcommand{\subsubsubsection}{\@startsection{paragraph}{4}{\z@}%
 {1.0\Cvs \@plus.5\Cdp \@minus.2\Cdp}%
 {.1\Cvs \@plus.3\Cdp}%
 {\reset@font\sffamily\normalsize}
 }
\makeatother
\DeclareMathOperator{\id}{id}

\DeclareMathOperator{\Hom}{Hom}

\DeclareMathOperator{\Aut}{Aut}

\DeclareMathOperator{\pr}{pr}

\DeclareMathOperator{\length}{length}
\DeclareMathOperator{\Gr}{Gr}
\DeclareMathOperator{\diag}{diag}
\DeclareMathOperator{\Res}{Res}
\DeclareMathOperator{\Sh}{Sh}
\DeclareMathOperator{\inv}{inv}
\DeclareMathOperator{\Par}{Par}

\DeclareMathOperator{\GL}{GL}

\DeclareMathOperator{\GU}{GU}

\DeclareMathOperator{\oF}{F}

\newcommand{\bA}{\mathbb{A}}
\newcommand{\bC}{\mathbb{C}}

\newcommand{\bF}{\mathbb{F}}

\newcommand{\bN}{\mathbb{N}}
\newcommand{\bP}{\mathbb{P}}
\newcommand{\bQ}{\mathbb{Q}}
\newcommand{\bR}{\mathbb{R}}
\newcommand{\bS}{\mathbb{S}}

\newcommand{\bX}{\mathbb{X}}
\newcommand{\bZ}{\mathbb{Z}}

\newcommand{\bfB}{\mathbf{B}}

\newcommand{\bfG}{\mathbf{G}}

\newcommand{\bfP}{\mathbf{P}}

\newcommand{\bfS}{\mathbf{S}}
\newcommand{\bfT}{\mathbf{T}}

\newcommand{\bfV}{\mathbf{V}}

\newcommand{\bfa}{\mathbf{a}}
\newcommand{\bfb}{\mathbf{b}}

\newcommand{\bft}{\mathbf{t}}

\newcommand{\cE}{\mathcal{E}}

\newcommand{\cL}{\mathcal{L}}

\newcommand{\cN}{\mathcal{N}}
\newcommand{\cO}{\mathcal{O}}

\newcommand{\cV}{\mathcal{V}}

\newcommand{\sG}{\mathscr{G}}

\newcommand{\sP}{\mathscr{P}}
\newcommand{\sQ}{\mathscr{Q}}

\newcommand{\sS}{\mathscr{S}}
\newcommand{\sT}{\mathscr{T}}
\newcommand{\sU}{\mathscr{U}}
\newcommand{\sV}{\mathscr{V}}
\newcommand{\sW}{\mathscr{W}}

\newcommand{\sfG}{\mathsf{G}}
\newcommand{\sfV}{\mathsf{V}}
\newcommand{\sfX}{\mathsf{X}}

\newcommand{\rL}{\mathrm{L}}

\newcommand{\pf}{\mathrm{pf}}

\newcommand{\ur}{\mathrm{ur}}

\newcommand{\ol}{\overline}

\newcommand{\wh}{\widehat}

\newcommand{\xra}{\xrightarrow}

\newcommand{\relmiddle}[1]{\mathrel{}\middle#1\mathrel{}}

\newcommand{\bMV}{\mathbb{MV}}

\newcommand{\Gm}{\mathbb{G}_{\mathrm{m}}}
\newcommand{\Ga}{\mathbb{G}_{\mathrm{a}}}
\newcommand{\Af}{\mathbb{A}_{\mathrm{f}}}

\newcommand{\cf}{\textit{cf.\ }}

\begin{document}

\title
{The supersingular locus of the Shimura variety\\ of $\mathrm{GU}(2,n-2)$} 

\author{Maria Fox and Naoki Imai}

\date{}
\maketitle

\footnotetext{2020 \textit{Mathematics Subject Classification}. 
 Primary: 11G18; Secondary: 14M15} 

\begin{abstract}
We study the 
supersingular locus of a reduction at an inert prime of 
the Shimura variety attached to $\mathrm{GU}(2,n-2)$. 
More concretely, we realize irreducible components of the supersingular locus 
as closed subschemes of flag schemes over Deligne--Lusztig varieties 
defined by explicit conditions after taking perfections. 
Moreover we study the intersections of the irreducible components. 
Stratifications of Deligne--Lusztig varieties defined using powers of 
Frobenius action appear in the description of the intersections. 
\end{abstract}

\section{Introduction}
Shimura varieties play an important role in the study of number theory. 
One way to approach the arithmetic of Shimura varieties is 
to construct integral models and study their reductions. 
The geometry of the supersingular locus of the reduction of a Shimura variety is especially useful information.
One of the striking results in this direction is the study of 
the supersingular locus of the reduction of the Shimura variety of $\GU (1,n-1)$ 
at an inert prime by Vollaard--Wedhorn in \cite{VoWeGUII}, 
where they give a description of 
the supersingular locus and their intersections in terms of 
Deligne--Lusztig varieties. 
This result is crucially used in \cite{KuRaSpuniI}. 

A long standing problem since \cite{VoWeGUII} 
is to extend such a result to unitary groups of 
other signatures. 
The only result in this line is the work \cite{HoPaGU22} of Howard--Pappas on the $\GU (2,2)$-case, which relies on an exceptional isomorphism.
A source of difficulty is that 
the Shimura variety of $\GU (2,n-2)$ is 
not fully Hodge--Newton decomposable 
in the sense of \cite[Definition 3.1]{GHNFulHN} 
if $n \geq 5$. 
In such a case, 
we can not expect that 
the supersingular locus is a union of Deligne--Lusztig varieties by \cite[Theorem B]{GHNFulHN}. 

On the other hand, 
the study of the perfection of the supersingular locus is essentially reduced to a study of an affine Deligne--Lusztig variety via 
the Rapoport--Zink uniformization. 
Further, a construction of irreducible components of an 
affine Deligne--Lusztig variety under some unramified condition is given by Xiao--Zhu in \cite{XiZhCycSh}. 
In their construction, we can rephrase the source of difficulty in the following way: 
Even though the affine Deligne--Lusztig variety 
related to the Shimura variety of $\GU (2,n-2)$ 
is defined using a minuscule cocharacter, 
non-minuscule cocharacters appear in 
the construction of 
its irreducible components if $n \geq 5$. 

The objective of this paper is 
to find an explicit description of the irreducible components of the affine Deligne--Lusztig variety related to 
the Shimura variety of $\GU (2,n-2)$ in terms of 
Deligne--Lusztig varieties. 

Let $F$ be a non-archimedean local field. 
We write $L$ for the completion of the maximal unramified extension of 
$F$. 
Let $G$ be the unramified general unitary group of degree $n$ 
over $F$. 
Let $\mu$ be the cocharacter of $G$ 
corresponding to 
$z \mapsto (\mathrm{diag}(z,z,1,\ldots,1),z)$ under an 
isomorphism $G_L \simeq \GL_n \times \Gm$. 
Let $X_{\mu^*}(\varpi^{-1})$ denote the affine Deligne--Lusztig variety 
for the dual $\mu^*$ of $\mu$ and $\varpi^{-1} \in G(L)$, 
where $\varpi$ is a uniformizer of $F$ and we regard $\varpi^{-1}$ as an element of $G(L)$ by embedding it into 
the $\Gm$-component. 
We put $r=[n/2]$. 
Then $X_{\mu^*}(\varpi^{-1})$ has $r$ isomorphism classes of irreducible components, whose representatives are given by $X_{\mu^*}^{\bfb_i,x_0}(\tau_i^*)$ for $1 \leq i \leq r$ as explained in \S \ref{sec:Irr}. 
If $i=1$ or $i=n/2$, then 
$X_{\mu^*}^{\bfb_i,x_0}(\tau_i^*)$ is isomorphic to the perfection of a Deligne--Lusztig variety as shown in 
Proposition \ref{prop:nu1} and Proposition \ref{prop:nur}. 

Assume that $2 \leq i \leq [(n-1)/2]$. 
Then the action of 
a hyperspecial subgroup $J_{\tau_i}(\cO_F) \subset G(F)$ 
on $X_{\mu^*}^{\bfb_i,x_0}(\tau_i^*)$ does not factor through the finite reductive quotient $J_{\tau_i}(\cO_F/\varpi)$ unlike the cases for $i=1,n/2$. 
We construct a kind of Demazure resolution 
$X_i$ of $X_{\mu^*}^{\bfb_i,x_0}(\tau_i^*)$. 
We write $\mathring{X}_i$ and $\mathring{X}_{\mu^*}^{\bfb_i,x_0}(\tau_i^*)$ for the inverse images in 
$X_i$ and $X_{\mu^*}^{\bfb_i,x_0}(\tau_i^*)$ 
of the Schubert cell $\mathring{\Gr}_{\nu_i^*}$ of an affine Grassmannian $\Gr_{\nu_i^*}$ under natural morphisms $X_i \to X_{\mu^*}^{\bfb_i,x_0}(\tau_i^*) \to  \Gr_{\nu_i^*}$. 
Explicitly, we construct a vector bundle 
$\sV_i$ of rank $2i-1$ over the perfection $Y_i$ 
of a Deligne--Lusztig variety. 
We have a natural morphism 
\[
 \phi_1 \colon \sV_i \to \oF (\sV_i^{\vee})
\]
by a Hermitian pairing related to the unitary group $G$, where 
$\oF (\sV_i^{\vee})$ is some Frobenius twist of $\sV_i^{\vee}$ (\cf \eqref{eq:phi1def}). 
Let $\Par_{t_i}(\sG_{Y_i})$ denote the flag scheme 
parametrizing subvector bundles $\sW \subset \sV_i$ 
of rank $i-1$. 

\begin{thm}[Theorem \ref{thm:pZi}, Proposition \ref{prop:Xopen}]\label{thm:IntpYi}
The scheme 
$X_i$ is isomorphic to the closed subscheme of 
$\Par_{t_i}(\sG_{Y_i})$ 
defined by the condition 
$\phi_1 (\sW) \subset \oF (\sW^{\perp})$ 
on $\sW$. 
Further $\mathring{X}_i$ is isomorphic to  $\mathring{X}_{\mu^*}^{\bfb_i,x_0}(\tau_i^*)$. 
\end{thm}

Let us summarize the situation in the following diagram: 
\[
 \xymatrix{
 \mathring{X}_{\mu^*}^{\bfb_i,x_0}(\tau_i^*)  \ar[r]^-{\sim}  & \mathring{X}_i         \ar[d] \ar@{^{(}->}[r] \ar@{}[dr]|\square & X_i \ar[d] \ar@{^{(}->}[r] &  \Par_{t_i}(\sG_{Y_i}) \ar[d] \\ 
  &  \mathring{\Gr}_{\nu_i^*} \ar@{^{(}->}[r]  & \Gr_{\nu_i^*} & Y_i .}
\]
By Theorem \ref{thm:IntpYi} and the above diagram, $\mathring{X}_{\mu^*}^{\bfb_i,x_0}(\tau_i^*)$ is cut out in $\Par_{t_i}(\sG_{Y_i})$ by two explicit conditions: one is a closed condition in Theorem \ref{thm:IntpYi} and the other is an open condition given by $\mathring{\Gr}_{\nu_i^*} \subset \Gr_{\nu_i^*}$. 

It is important to describe $X_i$, not only $\mathring{X}_i$, in order to study the intersections of irreducible components of $X_{\mu^*}(\varpi^{-1})$, because we need to understand a closure of $\mathring{X}_{\mu^*}^{\bfb_i,x_0}(\tau_i^*)$. 
We give a description of the intersections of the irreducible components in most cases in \S \ref{sec:Int}. 
Here we state one of the results, which exhibits an interesting new phenomenon. \begin{prop}[Proposition \ref{prop:inti2}]\label{prop:int12i}
The intersection 
$X_{\mu^*}^{\bfb_1,x_0}(\tau_1^*) \cap X_{\mu^*}^{\bfb_2,x_0}(\tau_2^*)$ 
is isomorphic to  
the perfect closed subscheme of $(\bP^{n-1})^{\mathrm{pf}}$ defined by two equations 
\[
 \sum_{i=1}^n x_i^{q+1}=0, \quad 
 \sum_{i=1}^n x_i^{q^3+1}=0. 
\]
\end{prop}
The perfect closed subscheme of $(\bP^{n-1})^{\mathrm{pf}}$ in 
Proposition \ref{prop:int12i} 
is the perfection of 
a stratification of a 
Deligne--Lusztig variety 
with respect to relative positions of 
parabolic subgroups and their twists 
by the third power of the Frobenius action. 
Such an intersection did not appear in the 
preceding research in 
fully Hodge--Newton decomposable cases. 
Our study does not cover all the intersections in general 
because of some technical difficulty 
which involves the study of vanishing of a ring with explicit generators and relations, 
but it does cover all the cases if $n \leq 6$. 

In the construction of 
irreducible components of an 
affine Deligne--Lusztig variety by Xiao--Zhu, they actually first construct 
$\mathring{X}_{\mu^*}^{\bfb_i,x_0}(\tau_i^*)$, and then construct 
$X_{\mu^*}^{\bfb_i,x_0}(\tau_i^*)$ 
as a closure of $\mathring{X}_{\mu^*}^{\bfb_i,x_0}(\tau_i^*)$. In the study of the unitary case in this paper, we clarify that this step in the construction is really necessary, i.e. 
we can not construct $X_{\mu^*}^{\bfb_i,x_0}(\tau_i^*)$ directly by a fiber product that is similar to the one used to construct $\mathring{X}_{\mu^*}^{\bfb_i,x_0}(\tau_i^*)$. 
This gives a negative answer to a question of Xiao--Zhu (\cf Remark \ref{rem:quesXZ}). 

The method in this paper should work for unitary groups of other signatures since the results in \cite{XiZhCycSh} and equidimensionality of Satake cycle in \S \ref{sec:EqScyc} 
are available also for other signatures. 
On the other hand, they will be more complicated for general signatures since 
the number of isomorphism classes of irreducible components of the affine Deligne--Lusztig varieties become larger. 
In this paper, we study the perfection of the supersingular locus via affine Deligne--Lusztig varieties. However, once the geometry of the corresponding affine Deligne--Lusztig varieties is understood using Demazure resolutions, we should be able to write a similar moduli problem using $p$-divisible groups and study them without taking perfections. 
That is a subject of \cite{FHIRZG}.

We explain the contents of each section. 
In \S \ref{sec:fla}, 
we recall a terminology on relative positions in flag schemes. 
We also give some gluing constructions of reductive group schemes.  
In \S \ref{sec:StDL}, 
we recall Deligne--Lusztig varieties and their Bruhat stratifications. 
We give also a new stratification 
using twists by a power of Frobenius map. 
We study the irreducibility of the stratification 
in some unitary case. 
In \S \ref{sec:AffG}, 
we recall affine Grassmannian and Satake cycles. 
In \S \ref{sec:EqScyc}, 
we recall and generalize 
results on equidimensionality of Satake cycles in 
\cite{HaiEquconv}. 
In \S \ref{sec:AffDL}, 
we recall a construction of irreducible components of 
affine Deligne--Lusztig varieties in \cite{XiZhCycSh}. 
In \S \ref{sec:Uni}, 
we explain the setting of a unitary group and 
apply the result in \S \ref{sec:EqScyc} to the unitary case.  
In \S \ref{sec:Irr}, 
we give an explicit description of irreducible components. 
In \S \ref{sec:Int}, 
we study the intersection of irreducible components. 
In \S \ref{sec:Exa}, 
we explain the results in the $n=6$ case as an example.   
In \S \ref{sec:ShVar}, 
we explain a relation between the 
affine Deligne--Lusztig varieties and the 
supersingular loci of reductions of Shimura varieties 
in our case. 

\subsection*{Acknowledgements}
The authors would like to thank Liang Xiao and Xinwen Zhu for answering questions on their work. The authors are grateful to Ben Howard and Ryosuke Shimada for helpful comments. They thank anonymous referees for their careful reading and helpful suggestions. 
The contents of this paper grew out of a discussion at the AIM workshop 
``Geometric realizations of Jacquet--Langlands correspondences'' in 2019. 
The authors are grateful to the organizers of the workshop for the invitations. Fox was partially supported by NSF MSPRF Grant 2103150. 
This work was supported by JSPS KAKENHI Grant Number 22H00093. 

\section{Flag scheme}\label{sec:fla}
\subsection{Relative position}
Let $\sG$ be a reductive group scheme over a scheme $\sS$. 
Let $\Par (\sG)$ be the $\sS$-scheme of parabolic subgroups of $\sG$. Let 
$\mathrm{Dyn}(\sG)$ be the $\sS$-scheme of Dynkin for $\sG$ 
constructed in \cite[XXIV, 3.3]{SGA3-3}. 
\begin{rem}
If $(\sT,M,R)$ is a splitting of $\sG$ in the sense of 
\cite[XXII, D\'efinition 1.13]{SGA3-3} and 
$\Delta$ is a set of simple roots, 
then we have a canonical isomorphism 
\begin{equation}\label{eq:DynDel}
 \mathrm{Dyn}(\sG) \simeq \Delta_{\sS} . 
\end{equation}
This is stated in \cite[XXIV, 3.4 (iii)]{SGA3-3} 
choosing a pinning, 
but the isomorphism actually depends only on 
$(\sT,M,R)$ and $\Delta$. 
\end{rem}

Let $\mathrm{Oc}(\mathrm{Dyn}(\sG))$ be the $\sS$-scheme of 
sets of open and closed subschemes of 
$\mathrm{Dyn}(\sG)$ (\cf \cite[XXVI, 3.1]{SGA3-3}). 
We have a projective smooth morphism 
\[
 \mathbf{t} \colon \Par (\sG) \to \mathrm{Oc}(\mathrm{Dyn}(\sG)) 
\]
of schemes as \cite[XXVI, Th\'eor\`eme 3.3]{SGA3-3}. For $t,t' \in \mathrm{Oc}(\mathrm{Dyn}(\sG))(\sS)$, we put 
\[
 \Par_{t} (\sG) =\bft^{-1} (t) \subset \Par (\sG), \quad  
 \Par_{t,t'} (\sG) =(\bft \times \bft)^{-1} (t,t') \subset \Par (\sG) \times_{\sS} \Par (\sG). 
\]
We recall results from \cite[XXVI. 4.5.3, 4.5.4]{SGA3-3}. 
Let $\mathrm{Stand}(\sG)$ be the 
$\sS$-scheme of pairs of parabolic subgroups of $\sG$ in mutually standard positions. 
Let $\mathrm{TypeStand}(\sG)$ be the $\sS$-scheme of types of mutually standard positions in $\sG$. 
The natural morphism 
\[
 \mathbf{t}_2 \colon \mathrm{Stand}(\sG) \to  \mathrm{TypeStand}(\sG), 
\] 
which is the quotient morphism under the conjugacy action of $\sG$, is smooth. 
There is a unique morphism 
\[
 q_{\sG} \colon \mathrm{TypeStand}(\sG) \to \mathrm{Oc}(\mathrm{Dyn}(\sG)) \times_{\sS} \mathrm{Oc}(\mathrm{Dyn}(\sG)) 
\]
such that the diagram 
\[
 \xymatrix{
  \mathrm{Stand}(\sG)  \ar[rr]^{\mathbf{t}_2}  \ar[d] &     &   \mathrm{TypeStand}(\sG)         \ar[d]^{q_{\sG}} \\ 
  \Par (\sG) \times_{\sS} \Par (\sG)   \ar[rr]^-{\mathbf{t} \times \mathbf{t}}   &  &  \mathrm{Oc}(\mathrm{Dyn}(\sG)) \times_{\sS} \mathrm{Oc}(\mathrm{Dyn}(\sG))} 
\]
is commutative. 
Let $\sP$ be a parabolic subgroup scheme of $\sG$. 
Let $\Par (\sG; \sP)$ be the $\sS$-scheme of parabolic subgroups of $\sG$ in standard positions relative to $\sP$. 
Let $t \in \mathrm{Oc}(\mathrm{Dyn}(\sG))(\sS)$. 
We put 
\[
 \Par_t (\sG;\sP)=\Par (\sG;\sP) \cap \Par_t (\sG). 
\]
Then we have a morphism 
\[
 \mathbf{t}_{\sP} \colon \Par_t (\sG;\sP) \to q_{\sG}^{-1}(\mathbf{t}(\sP),t)  
\]
induced by $\bft_2$. 
For an $\sS$-scheme $\sS'$ and 
$r \in (q_{\sG}^{-1}(\mathbf{t}(\sP),t))(\sS')$, 
we define $\Par_t (\sG;\sP)_r$ by the fiber product 
\[
 \xymatrix{
 \Par_t (\sG;\sP)_r   \ar[r] \ar[d]     &   \sS'         \ar[d]^{r} \\ 
  \Par_t (\sG;\sP)  \ar[r]^-{\mathbf{t}_{\sP}}   &  q_{\sG}^{-1}(\mathbf{t}(\sP),t) .} 
\]
\begin{rem}
Let $\sQ$ be a parabolic subgroup scheme of $\sG$. 
Let $\sS'$ be an $\sS$-scheme. 
We write $\sG'$, $\sP'$, $\sQ'$ for the base change of $\sG$, $\sP$, $\sQ$ to $\sS'$. 
Assume that a maximal torus 
$\sT'$ of $\sG'$ is contained in $\sP' \cap \sQ'$. 
Then we have a natural isomorphism 
\begin{equation}\label{eq:WTS}
 W_{\sP'}(\sT') \backslash W_{\sG'}(\sT') / W_{\sQ'}(\sT') 
 \simeq q_{\sG}^{-1}(\mathbf{t}(\sP),\mathbf{t}(\sQ)) \times_{\sS} \sS' 
\end{equation}
over $\sS'$ as in \cite[XXVI. 4.5.3]{SGA3-3}, 
where $W_{\sP'}(\sT')$, $W_{\sG'}(\sT')$ and $W_{\sQ'}(\sT')$ are the Weyl groups defined as \cite[XII, 2]{SGA3-2}. 
\end{rem}

\begin{notn}
Assume that $\sG$ is split and $\sS$ is connected. 
Let $(\sT,M,R)$ be a splitting of $\sG$ 
and $\Delta$ be a set of simple roots. 
Let $(W,S)$ be the Coxeter system of $(M,R,\Delta)$. 
For $I \subset S$, let 
$W_I$ be the subgroup of $W$ generated by $I$, 
and 
let $t(I)$ be the element of 
$\mathrm{Oc}(\mathrm{Dyn}(\sG))(\sS)$ 
corresponding to $I$ under \eqref{eq:DynDel}. 
Conversely, 
let $I(t)$ be the subset of $S$ 
corresponding to $t$ under \eqref{eq:DynDel} 
for $t \in \mathrm{Oc}(\mathrm{Dyn}(\sG))(\sS)$. 
We simply write 
$W_t$ for $W_{I(t)}$. 
\end{notn}

\subsection{Inner gluing}
\begin{defn}
Let $\sG_0$ be a reductive group scheme over a scheme $\sS_0$. 
Let $\sS$ be a scheme over $\sS_0$. 
An inner gluing over $\sS$ of $\sG_0$ is 
a pair $(\sG , \varphi)$, where $\sG$ is a reductive group scheme over $\sS$ and 
$\varphi$ is a global section of the Zariski sheaf 
\[
 \underline{\mathrm{Isom}}_{\sS}(\sG_0 \times_{\sS_0} \sS,\sG )/ 
 \underline{\mathrm{Inn}}_{\sS}(\sG_0 \times_{\sS_0} \sS) 
\]
on $\sS$. 
\end{defn}

\begin{rem}\label{rem:vecig}
Let $\sV$ be a vector bundle of rank $n$ on $\sS$. 
We put $\sG=\Aut_{\sS} (\sV)$. 
By taking Zariski local trivializations of 
$\sV$, we obtain an inner gluing $(\sG,\varphi_{\sV})$ over $\sS$ of $\GL_{n,\bZ}$. This is independent of the choice of trivializations, because a difference of trivializations induces an inner automorphism of $\GL_n$.  
\end{rem}

\begin{lem}\label{lem:igiso}
Let $\pi \colon \sS \to \sS_0$ be a morphism of schemes. 
Let $\sG_0$ be a reductive group scheme over $\sS_0$. 
Let $(\sG , \varphi)$ 
an inner gluing over $\sS$ of $\sG_0$. 
\begin{enumerate}
\item\label{en:Dyphi}
The section $\varphi$ induces isomorphisms 
\begin{align*}
 \mathrm{Oc}(\mathrm{Dyn}(\sG_0)) \times_{\sS_0} \sS &\xrightarrow{\sim}  \mathrm{Oc}(\mathrm{Dyn}(\sG)), \\ 
 \mathrm{TypeStand}(\sG_0) \times_{\sS_0} \sS &\xrightarrow{\sim}  \mathrm{TypeStand}(\sG) 
\end{align*}
which are compatible with $q_{\sG_0}$ and $q_{\sG}$. 
\item\label{en:Wphi}
Assume that $\sG_0$ is split and $\sS_0$ is connected. 
Let $(\sT_0,M,R)$ be a splitting of $\sG_0$ 
and $\Delta$ be a set of simple roots. 
Let $(W,S)$ be the Coxeter system of $(M,R,\Delta)$. 
Let $t_0,t_0' \in \mathrm{Oc}(\mathrm{Dyn}(\sG_0))(\sS_0)$. 
Let $t,t' \in \mathrm{Oc}(\mathrm{Dyn}(\sG))(\sS)$ 
denote the pullbacks to $\sS$ of $t_0,t_0'$. 
Then $\varphi$ induces an isomorphism 
\[
 (W_{t_0} \backslash W /W_{t_0'})_{\sS} \xrightarrow{\sim} 
 q_{\sG}^{-1}(t,t'). 
\]
\end{enumerate}
\end{lem}
\begin{proof}
There exist a Zariski covering $\{ \sU_{\lambda} \}_{\lambda \in \Lambda}$ of $\sS$ and 
a family of isomorphisms $\varphi_{\lambda} \colon \sG_0 \times_{\sS_0} \sU_{\lambda} \xrightarrow{\sim} \sG \times_{\sS} \sU_{\lambda}$ such that 
$\varphi_{\lambda}$ is compatible with 
$\varphi|_{\sU_{\lambda}}$.  
Then the family of isomorphisms $\varphi_{\lambda}$ 
induces isomorphisms 
\[
 \mathrm{Oc}(\mathrm{Dyn}(\sG_0)) \times_{\sS_0} \sU_{\lambda} \xrightarrow{\sim}  \mathrm{Oc}(\mathrm{Dyn}(\sG \times_{\sS} \sU_{\lambda})). 
\]
These isomorphisms glue together to give the first isomorphism in the claim \ref{en:Dyphi} by \cite[XXIV, 3.4 (iv)]{SGA3-3}. 

The family of isomorphisms $\varphi_{\lambda}$ 
induce also isomorphisms 
\[
 \mathrm{Stand}(\sG_0) \times_{\sS_0} \sU_{\lambda} \xrightarrow{\sim}  \mathrm{Stand}(\sG \times_{\sS} \sU_{\lambda}). 
\]
By taking the quotients by the conjugacy actions of 
$\sG_0 \times_{\sS_0} \sU_{\lambda} \simeq  \sG \times_{\sS} \sU_{\lambda}$, we obtain 
isomorphisms 
\[
 \mathrm{TypeStand}(\sG_0) \times_{\sS_0} \sU_{\lambda} \xrightarrow{\sim}  \mathrm{TypeStand}(\sG \times_{\sS} \sU_{\lambda}). 
\]
These isomorphisms glue together to give the second isomorphism in the claim \ref{en:Dyphi} 
because we take quotients by conjugacy actions. 
By the constructions, two isomorphisms in 
the claim \ref{en:Dyphi} 
are compatible with $q_{\sG_0}$ and $q_{\sG}$. 

By \ref{en:Dyphi}, we have an isomorphism 
\begin{equation}\label{eq:q-1bc}
 q_{\sG_0}^{-1}(t_0,t_0') \times_{\sS_0} \sS \xrightarrow{\sim} 
 q_{\sG}^{-1}(t,t') 
\end{equation}
induced by $\varphi$. 
The claim \ref{en:Wphi} follows from \cite[XXII, Proposition 3.4]{SGA3-3} and \eqref{eq:q-1bc}. 
\end{proof}

\section{Stratification of Deligne--Lusztig variety}\label{sec:StDL}
\subsection{Deligne--Lusztig variety}
Let $\bfG_0$ be a connected reductive group over $\bF_q$. 
We take a maximal torus and a Borel subgroup 
$\bfT_0 \subset \bfB_0 \subset \bfG_0$ over $\bF_q$. 
We write 
$\bfG$, $\bfB$ and $\bfT$ for the base changes to 
$\ol{\bF}_q$ of $\bfG_0$, $\bfB_0$ and $\bfT_0$. 
Let $(W,S)$ be the Coxeter system of $\bfG$ with respect to $\bfT$ and $\bfB$. 
For $I, J \subset S$, we write 
$\Par_I (\bfG)$ and $\Par_{I,J} (\bfG)$ 
for 
$\Par_{t(I)} (\bfG)$ and $\Par_{t(I),t(J)} (\bfG)$. 

For $I, J \subset S$ and  $w \in W$, 
we put 
\[
 \Par_{I,J} (\bfG)_{[w]}=\bft_2^{-1}(r_w), 
\]
where $r_w \in (q_{\bfG}^{-1}(t(I),t(J)))(\ol{\bF}_q)$ 
corresponds to $[w] \in W_I \backslash W /W_J$ by 
Lemma \ref{lem:igiso} \ref{en:Wphi}. 
Let $\Par_{I,J} (\bfG)_{\leq [w]}$ be the closed reduced subscheme of $\Par_{I,J} (\bfG)$ 
determined by 
\[
 \bigcup_{[w'] \leq [w]} \Par_{I,J} (\bfG)_{[w']}. 
\]
Let $\oF$ be the $q$-th power Frobenius endomorphism of 
$\bfG$ obtained from $\bfG_0$. 
Let $I \subset S$ and  $w \in W$. 
For $* \in \{ [w], {\leq}[w] \}$ 
with $[w] \in W_I \backslash W /W_{\oF (I)}$, 
let $X_I^{\oF}(*)$ 
be the locally closed subscheme of $\Par_I (\bfG)$ 
defined by the fiber product 
\[
 \xymatrix{
 X_I^{\oF}(*)  \ar[r] \ar[d] &   \Par_{I,\oF (I)} (\bfG)_{*}         \ar[d] \\ 
  \Par_I (\bfG)  \ar[r]^-{(\id ,\oF)}   &  \Par_I (\bfG) \times \Par_{\oF (I)} (\bfG) .} 
\]
If $\bfP_I$ is a parabolic subgroup of 
$\bfG$ of type $I$ containing $\bfB$, then $X_I^{\oF}(*)$ is identified with 
\[
 \{ g \bfP_I \in \bfG/\bfP_I \mid g^{-1}\oF (g) \in \bfP_I\backslash \bfG / \bfP_{F(I)} \cong W_I \backslash W /W_{\oF (I)} \ \textrm{is} \ *\}
\]
under the isomorphism $\Par_I (\bfG) \simeq \bfG/\bfP_I$. 
If there is no confusion, we simply write $X_I(*)$ for $X_I^{\oF}(*)$. See \cite[\S 4.4]{VoWeGUII} for general properties of $X_I([w])$. 

For $I \subset J \subset S$, we have a natural morphism 
\[
 \pi_{I,J} \colon X_I([w]) \to X_J([w])
\] 
which sends a 
parabolic subgroup $\bfP$ of $\bfG$ of type $I$ 
to a unique parabolic subgroup $\bfP'$ of $\bfG$ of type $J$ containing $\bfP$. 

\subsection{Bruhat stratification}
Let $I, J \subset S$ and  $w \in W$. 
Let $\bfP_J$ be a parabolic subgroup of 
$\bfG$ of type $J$. 
For $* \in \{ [w'], {\leq}[w'] \}$ 
with $[w'] \in W_I \backslash W /W_J$, 
we let $X_I([w])_{\bfP_J,*}$ 
be the locally closed subscheme of $X_I([w])$ 
defined by the fiber product 
\[
 \xymatrix{
 X_I([w])_{\bfP_J,*}  \ar[r] \ar[d] &   \Par_{I,J} (\bfG)_{*}         \ar[d] \\ 
  X_I([w])  \ar[r]^-{(\id ,\bfP_J)}   &  \Par_I (\bfG) \times \Par_J (\bfG) .} 
\]
If $\bfP_I$ and $\bfP_J$ are parabolic subgroups of 
$\bfG$ of type $I$ and $J$ containing $\bfB$ respectively, then $X_I([w])_{\bfP_J,*}$ is identified with 
\[
 \{ g \bfP_I \in X_I([w]) \subset \bfG/\bfP_I \mid g^{-1} \in \bfP_I\backslash \bfG / \bfP_J \simeq W_I \backslash W /W_J \ \textrm{is} \ *\}. 
\]

\subsection{Stratification relative to Frobenius twists}
For $1 \leq i \leq m$, 
let $\oF_i$ be a Frobenius endomorphism of 
$\bfG$ which descends it to an algebraic group over a finite field. 
Let $w_1,\ldots ,w_m \in W$. 
For $*_i \in \{ [w_i], {\leq}[w_i] \}$ 
with $[w_i] \in W_I \backslash W /W_{\oF_i(I)}$ and $1 \leq i \leq m$, 
let $X_{I}^{\oF_1,\ldots,\oF_m}(*_1,\ldots,*_m)$ 
be the locally closed subscheme of $\mathrm{Par}_I$ 
defined by the fiber product 
\[
 \xymatrix{
 X_{I}^{\oF_1,\ldots,\oF_m}(*_1,\ldots,*_m) \ar[rr] \ar[d] & &  \prod_{1 \leq i \leq m} \Par_{I,\oF_i(I)} (\bfG)_{*_i}         \ar[d] \\ 
  \Par_I (\bfG)  \ar[rr]^-{\prod_{1 \leq i \leq m} (\id ,\oF_i)}   & & \prod_{1 \leq i \leq m} \left( \Par_I (\bfG) \times \Par_{\oF_i(I)} (\bfG) \right).} 
\]
Then 
$X_{I}^{\oF_1,\ldots,\oF_m}([w_1],\ldots,[w_m])$ for $[w_i] \in W_I \backslash W /W_{\oF_i(I)}$ and $2 \leq i \leq m$ 
give a stratification of 
$X_{I}^{\oF_1}([w_1])$. 
We note that 
$X_{I}^{\oF_1,\ldots,\oF_m}([w_1],\ldots,[w_m])=\bigcap_{1 \leq i \leq m} X_{I}^{\oF_i}([w_i])$ by the definition. 

\subsection{Unitary case}
We put $\bfV_0=\bF_{q^2}^d$ equipped with the hermitian form 
\begin{equation}\label{eq:Herpair}
 \bF_{q^2}^d \times \bF_{q^2}^d \to \bF_{q^2};\ 
 ((a_i)_{1 \leq i \leq d},(a_i')_{1 \leq i \leq d}) 
 \mapsto \sum_{i=1}^d a_i^q a_{d+1-i}' . 
\end{equation}
We put $\bfG_0=\GU (\bfV_0)$. 
By taking the first factor of the decomposition 
\[
 \bF_{q^2} \otimes_{\bF_q} \bF_{q^2} \simeq \bF_{q^2} \times \bF_{q^2};\ 
 a \otimes b \mapsto (ab,ab^q) , 
\]
we have an isomorphism 
\begin{equation}\label{eq:funisp}
 \bfG \simeq \GL_d \times \Gm . 
\end{equation}
Let $\bfT \subset \bfB \subset \bfG$ be the maximal torus and the Borel subgroup determined by 
the diagonal torus $T_d$ and the upper triangular subgroup $B_d$ of $\GL_d$ 
under \eqref{eq:funisp}. 
Let $(W_{\bfG},\{ s_1,\ldots ,s_{d-1} \})$ 
be the Coxeter system of $\bfG$ with respect to $\bfT$ and $\bfB$, 
where $s_i$ corresponds to the simple root 
\[
 T_d \times \Gm \to \Gm ;\ (\diag (x_1,\ldots,x_d),z) 
 \mapsto x_i x_{i+1}^{-1} 
\]
of $\GL_d \times \Gm$ under \eqref{eq:funisp}. 
For $1 \leq i_1 < \cdots < i_l \leq d-1$, 
we put 
\[
 I_{d}^{i_1,\ldots,i_l} = \{ s_i \}_{i \in \{ 1, \ldots, d-1 \} \setminus \{ i_1,\ldots,i_l\} }. 
\]

\begin{eg}\label{eq:Fermat}
By the correspondence between parabolic subgroups of $\GL_d$ of type $I_{d}^{1}$ and lines in $\bf{V}_0$, the scheme $X_{I_{d}^{1}}([1])$ parametrizes lines $L$ in $\bf{V}_0$ such that $L \subset \bf{V}_0$ is contained in $\oF (L^{\perp}) \subset \oF (\bf{V}_0^{\vee})$ under the identification $\bf{V}_0 \simeq \oF (\bf{V}_0^{\vee})$ given by the pairing \eqref{eq:Herpair}. 
Writing the coordinates of $L \in \bP^{d-1}$ as $(x_1,\ldots , x_d)$, we see that $X_{I_{d}^{1}}([1])$ is isomorphic to the Fermat hypersurface defined by 
\[
\sum_{i=1}^d x_i x_{d+1-i}^q=0 
\]
in $\bP^{d-1}$. 
\end{eg}

\begin{lem}\label{lem:irred}
Assume that $2 \leq i \leq d/2$. The schemes 
$X_{I_{d}^{i-1}}^{\oF,\oF^2,\oF^3}([1],{\leq}[s_{i-1}],[1])$ and 
$X_{I_{d}^{d-i}}^{\oF,\oF^2}([1],{\leq}[s_{d-i}])$ are irreducible. 
\end{lem}
\begin{proof}
The scheme $X_{I_{d}^{i-1,d-i}}([1])$ 
is irreducible by \cite[Theorem 1]{BoRoIrrDL}. 
Hence, it suffices to show the following claims: 
\begin{enumerate}
\item \label{en:imi-1}
The image of 
\[
 \pi_{I_{d}^{i-1,d-i},I_{d}^{i-1}} \colon 
 X_{I_{d}^{i-1,d-i}}([1]) \to X_{I_{d}^{i-1}}([1])
\]
on $\ol{\bF}_{q}$-valued points is equal to 
$X_{I_{d}^{i-1}}^{\oF,\oF^2,\oF^3}([1],{\leq}[s_{i-1}],[1])(\ol{\bF}_{q})$.
\item \label{en:imd-i}
The image of 
\[
 \pi_{I_{d}^{i-1,d-i},I_{d}^{d-i}} \colon 
 X_{I_{d}^{i-1,d-i}}([1]) \to X_{I_{d}^{d-i}}([1])
\]
on $\ol{\bF}_{q}$-valued points is equal to  $X_{I_{d}^{d-i}}^{\oF,\oF^2}([1],{\leq}[s_{d-i}])(\ol{\bF}_{q})$. 
\end{enumerate}
We show the claim \ref{en:imi-1}. 
We equip $\ol{\bF}_{q}^d$ with the pairing 
\begin{equation}\label{eq:olpari}
 \ol{\bF}_q^d \times \ol{\bF}_q^d \to \ol{\bF}_q;\ 
 ((x_i)_{1 \leq i \leq d},(y_i)_{1 \leq i \leq d}) 
 \mapsto \sum_{i=1}^d x_i y_{d+1-i} .     
\end{equation}
For an $\ol{\bF}_q$-vector subspace 
$V \subset \ol{\bF}_q^d$, let 
$V^{\perp}$ denote the orthogonal complement of $V$ with respect to the pairing \eqref{eq:olpari}. 
The $q$-th power Frobenius element $\oF$ acts on 
$\ol{\bF}_{q}^d$. 
A point of $X_{I_{d}^{i-1,d-i}}([1])(\ol{\bF}_q)$ 
corresponds to a filtration 
$0 \subset V_1 \subset V_2 \subset \ol{\bF}_{q}^d$ 
such that $\dim V_1 =i-1$, $\dim V_2 =d-i$ 
and 
\begin{equation}\label{eq:fil12}
 V_1 \subset \oF (V_2^{\perp}) \subset V_2 \subset \oF (V_1^{\perp}). 
\end{equation}
The condition \eqref{eq:fil12} implies 
\begin{equation}\label{eq:V1F2}
 V_1 +\oF^2(V_1) \subset \oF (V_2^{\perp}). 
\end{equation}
Therefore we have 
\begin{equation}\label{eq:F3V1}
 \oF^3(V_1) \subset \oF (V_1 +\oF^2(V_1)) \subset \oF^2(V_2^{\perp}) \subset 
 \oF (V_2) \subset V_1^{\perp} \cap \oF^2(V_1^{\perp}) \subset V_1^{\perp}. 
\end{equation}
The conditions 
\eqref{eq:fil12}, \eqref{eq:V1F2} 
and \eqref{eq:F3V1} imply that 
$V_1$ defines a point of 
\begin{equation*}\label{eq:ptXF123}
 X_{I_{d}^{i-1}}^{\oF,\oF^2,\oF^3}([1],{\leq}[s_{i-1}],[1])(\ol{\bF}_{q}), 
\end{equation*}
because $\dim \oF (V_2^{\perp})=i$. 
To show the claim \ref{en:imi-1}, 
it suffices to show that the image of $\pi_{I_{d}^{i-1,d-i},I_{d}^{i-1}}$ 
on $\ol{\bF}_{q}$-valued points contains 
\begin{equation}\label{eq:ptXF123op}
X_{I_{d}^{i-1}}^{\oF,\oF^2,\oF^3}([1],[s_{i-1}],[1])(\ol{\bF}_{q}), 
\end{equation}
because 
$X_{I_{d}^{i-1,d-i}}([1])$ is proper. 
A point of \eqref{eq:ptXF123op} gives 
an $\ol{\bF}_q$-vector subspace 
$V_1 \subset \ol{\bF}_q^d$ of dimension $i-1$ such that 
\begin{align}\label{eq:cdV1op}
 V_1 \subset \oF (V_1^{\perp}), \quad 
 \dim (V_1 +\oF^2 (V_1)) = i, \quad 
 \oF^3(V_1) \subset V_1^{\perp}. 
\end{align}
The condition 
implies 
\[
 \oF (V_1 +\oF^2(V_1)) \subset V_1^{\perp} \cap \oF^2(V_1^{\perp}) 
\]
and $\dim V_1^{\perp} \cap \oF^2(V_1^{\perp})=d-i$. 
We take $V_2 \subset \ol{\bF}_q^d$ such that 
$\oF (V_2) =V_1^{\perp} \cap \oF^2(V_1^{\perp})$. 
Then $(V_1,V_2)$ defines a point of 
$X_{I_{d}^{i-1,d-i}}([1])(\ol{\bF}_q)$ whose image under 
$\pi_{I_{d}^{i-1,d-i},I_{d}^{i-1}}$ is the point of 
\eqref{eq:ptXF123op} corresponding to $V_1$. 
Therefore we obtain the claim \ref{en:imi-1}. 

The claim \ref{en:imd-i} is proved similarly. 
\end{proof}

\section{Affine Grassmannian}\label{sec:AffG}

Let $F$ be a non-archimedean local field with residue field $k=\bF_q$. Let $\cO_F$ be the ring of integers of $F$.
Let $\varpi$ be a uniformizer of $F$. 
For a perfect $k$-algebra $R$, 
we put 
\[
 W_{\cO_F}(R)=\varprojlim_{n} W(R) \otimes_{W(k)} \cO_F /\varpi^n , 
\]
$D_R = \mathrm{Spec}(W_{\cO_F}(R) )$ and $D^*_R = \mathrm{Spec}(W_{\cO_F}(R)[\frac{1}{\varpi}] )$. 
For an affine group scheme $H$ of finite type over $\cO_F$, we define the jet group 
$\rL^+H$ and the loop group $\rL H$ by 
\[
 \rL^+H(R)=H(W_{\cO_F}(R)), \quad 
 \rL H(R)=H(W_{\cO_F}(R)[\frac{1}{\varpi}]). 
\]
We put $L=W_{\cO_F}(\ol{k})[\frac{1}{\varpi}]$. We note that 
$\rL H(\ol{k})=H(L)$. 

Let $G$ be a reductive group scheme over $\cO_F$. 
Let $T$ be the abstract Cartan subgroup of $G$. 
Let $\Phi \subset \bX^{\bullet}( T)$ denote the set of roots of $G$ in the weight lattice, and 
$\bX_{\bullet}( T)^+ \subset \bX_{\bullet}( T)$ the semi-group of dominant coweights in the coweight lattice. 
Let $\rho \in \bX^{\bullet}(T)_{\bQ}$ be the half sum of all positive roots. 
We fix a Borel subgroup $B \subset G$. 
Let $U$ be the unipotent radical of $B$. 
Then $T$ is canonically identified with $B/U$. 

Let $\Gr_G$ denote the affine Grassmannian over $k$ of $G$ 
defined by $\Gr_G = \rL G/\rL^+ G$. 
For a finite etale extension $\cO'$ of $\cO_F$ with residue field $k'$, 
we have a natural isomorphism 
\begin{equation}\label{eq:Grbc}
 (\Gr_G)_{k'} \simeq \Gr_{G_{\cO'}} 
\end{equation}
by the construction. 
We simply write $\Gr$ for $\Gr_G$ if there is no confusion. 
Then $\Gr$ is an ind-perfectly projective scheme 
by \cite[Corollary 9.6]{BhScProWG}. 
Let $\cE^0$ denote the trivial $G$-torsor over $\cO_F$. 
For a perfect $k$-algebra $R$, we have 
\begin{equation}\label{eq:Grmod}
 \Gr(R) =  \left\{ 
 (\cE, \beta) \relmiddle{|} 
 \begin{tabular}{ l }
  $\cE$ is a $G$-torsor on $D_R$,  \\ 
  $\beta \colon \cE|_{D^*_R} \simeq \cE^0|_{D^*_R}$ is a trivialization 
  \end{tabular} 
\right\}    
\end{equation}
(\cf \cite[Lemma 1.3]{ZhuAffGmix}). 
We sometimes write $\beta \colon \cE \dashrightarrow \cE^0$ for 
$\beta \colon \cE|_{D^*_R} \simeq \cE^0|_{D^*_R}$ in \eqref{eq:Grmod}, 
and call it a modification.  
Given a point $(\cE, \beta)$, one can define a relative position invariant $\inv (\beta) \in \bX_{\bullet}( T)^+$.

Let $\mu \in \bX_{\bullet}( T)^+$. 
The Schubert variety $\Gr_{\mu}$ 
is the closed subscheme of $\Gr_{\ol{k}}$ 
parametrizing pairs $(\cE, \beta)$ such that 
$\inv (\beta) \preceq \mu$. 
The Schubert cell $\mathring{\mathrm{Gr}}_{\mu}$ 
is the open subscheme of $\Gr_{\mu}$ 
parametrizing pairs $(\cE, \beta)$ such that 
$\inv (\beta) = \mu$. 

For a sequence $\mu_{\bullet} =(\mu_1, \ldots, \mu_n)$ of dominant coweights, let $\Gr_{\mu_{\bullet}}$  be the scheme over $\ol{k}$ parametrizing sequences of modifications 
$( \beta_i \colon \cE_i \dashrightarrow \cE_{i-1})_{1 \leq i \leq n}$ with 
$\cE_0=\cE^0$ 
such that $\inv (\beta_i) \preceq \mu_i$ for each $i$. 
The open subscheme 
$\mathring{\Gr}_{\mu_{\bullet} } \subset \Gr_{\mu_{\bullet} }$ 
is defined by the condition that 
$\inv (\beta_i) = \mu_i$ for each $i$. 
The convolution map $m_{\mu_{\bullet}} \colon \Gr_{\mu_{\bullet}} \rightarrow \Gr_{\ol{k}}$ sends a sequence of modifications to the composition $(\cE_n, \beta_1 \circ \cdots \circ \beta_n)$. 

Let $\lambda_{\bullet}=(\lambda_1, \ldots, \lambda_l)$ and 
$\mu_{\bullet}=(\mu_1, \ldots, \mu_n)$ be two sequences. 
We put 
\[
 \Gr^0_{\lambda_\bullet | \mu_{\bullet}} = \Gr_{\lambda_{\bullet} }  \times_{\Gr_{\ol{k}}} 
 \Gr_{\mu_{\bullet} }, \quad 
 \mathring{\Gr}^0_{\lambda_\bullet | \mu_{\bullet}} = \mathring{\Gr}_{\lambda_{\bullet} }  \times_{\Gr_{\ol{k}}} 
 \mathring{\Gr}_{\mu_{\bullet} }, 
\]
where the products are over the convolution maps $m_{\lambda_{\bullet}} \colon \Gr_{\lambda_{\bullet} } \rightarrow \Gr_{\ol{k}}$, $m_{\mu_{\bullet}} \colon \Gr_{\mu_{\bullet} } \rightarrow \Gr_{\ol{k}}$ and their restrictions respectively. 
We write 
\[
 m_{\lambda_\bullet | \mu_{\bullet}} \colon 
 \Gr^0_{\lambda_\bullet | \mu_{\bullet}} \to \Gr_{\ol{k}} 
\]
for the natural projection. 
We simply write $m$ for $m_{\lambda_\bullet | \mu_{\bullet}}$ if there is no confusion.  
For $1 \leq j \leq l$, we define 
\[
 \pr_j \colon \Gr^0_{\lambda_\bullet | \mu_{\bullet}} \to 
 \Gr_{(\lambda_1,\ldots,\lambda_j)} 
\]
by sending $((\alpha_i)_{1 \leq i \leq l},(\beta_i)_{1 \leq i \leq n})$ 
to $(\alpha_i)_{1 \leq i \leq j}$. 

An irreducible component of 
$\Gr^0_{\lambda_{\bullet} | \mu_{\bullet}}$ 
of dimension $\langle \rho , \lvert \lambda_{\bullet} \rvert +\lvert \mu_{\bullet} \rvert \rangle$ 
is called a Satake cycle. 
Let $\bS_{\lambda_\bullet | \mu_{\bullet}}$ 
be the set of Satake cycles in 
$\Gr^0_{\lambda_\bullet | \mu_{\bullet}}$. 
We sometimes write 
$\Gr^{0,\bfa}_{\lambda_\bullet | \mu_{\bullet}}$ instead of $\bfa \in \bS_{\lambda_\bullet | \mu_{\bullet}}$ 
for the Satake cycle. 
We put 
\[
 \mathring{\Gr}^{0,\bfa}_{\lambda_\bullet | \mu_{\bullet}} = 
 \Gr^{0,\bfa}_{\lambda_\bullet | \mu_{\bullet}} \cap 
 \mathring{\Gr}^0_{\lambda_\bullet | \mu_{\bullet}}. 
\]
\begin{lem}\label{lem:SatCyCl}
For $\bfa \in \bS_{\lambda_\bullet | \mu_{\bullet}}$, the scheme 
$\mathring{\Gr}^{0,\bfa}_{\lambda_\bullet | \mu_{\bullet}}$ 
is not empty. 
\end{lem}
\begin{proof}
The dimension of 
$\Gr^0_{\lambda_\bullet | \mu_{\bullet}} \setminus \mathring{\Gr}^0_{\lambda_\bullet | \mu_{\bullet}}$ 
is less than $\langle \rho , \lvert \lambda_{\bullet} \rvert +\lvert \mu_{\bullet} \rvert \rangle$ by \cite[Proposition 3.1.10 (1)]{XiZhCycSh}. Hence we obtain the claim. 
\end{proof}

We fix an embedding 
$T \subset B$. Let $\mu \in \bX_{\bullet}( T)$. 
Let $\cO'$ be $\cO_L=W_{\cO_F}(\ol{k})$ or a finite etale extension of $\cO_F$ 
which splits $G$. 
For $\alpha \in \Phi$, let $U_{\alpha,\cO'}$ denote the root subgroup of 
$G_{\cO'}$ corresponding to $\alpha$. 
Let $P_{\mu,\cO'}$ denote the parabolic subgroup of 
$G_{\cO'}$ generated by $T_{\cO'}$ and $U_{\alpha,\cO'}$ for 
$\alpha \in \Phi$ such that $\langle \alpha , \mu \rangle \geq 0$.

We write $\varpi^\mu$ for 
$\mu (\varpi) \in G(L)=\rL G (\ol{k})$. 
Let $[\varpi^{\mu}]$ denote the point of $\Gr_{\ol{k}}$ determined by $\varpi^{\mu}$. 
For $\mu \in \bX_{\bullet}( T)^+$, 
the Schubert cell $\mathring{\mathrm{Gr}}_{\mu}$ is the $\rL^+ G$-orbit of 
$[\varpi^{\mu}]$ by \cite[Proposition 1.23 (1)]{ZhuAffGmix}. 

\begin{lem}\label{lem:SatSurj}
For $\bfa \in \bS_{\lambda_\bullet | \mu}$, the natural morphism 
$\Gr^{0,\bfa}_{\lambda_\bullet | \mu} \to  
 \Gr_{\mu}$ is surjective. 
\end{lem}
\begin{proof}
The natural morphism 
$\mathring{\Gr}^{0,\bfa}_{\lambda_\bullet | \mu} \to  \mathring{\Gr}_{\mu}$ is surjective, 
because the action of $\rL^+ G$ on 
$\mathring{\mathrm{Gr}}_{\mu}$ is transitive and 
$\mathring{\Gr}^{0,\bfa}_{\lambda_\bullet | \mu}$ 
is a nonempty scheme stable under the action of 
$\rL^+ G$ by Lemma \ref{lem:SatCyCl}. 
Hence we obtain the claim 
because 
$\Gr^{0}_{\lambda_\bullet | \mu} \to  
 \Gr_{\mu}$ is perfectly proper and 
$\mathring{\Gr}_{\mu} \subset \Gr_{\mu}$ is Zariski dense by \cite[Proposition 1.23 (3)]{ZhuAffGmix}. 
\end{proof}

For $\lambda \in \bX_{\bullet}(T)$, 
let 
$S_{\lambda}$ 
be the 
$(\rL U)_{\ol{k}}$-orbit of $\varpi^{\lambda}$ in $\Gr_{\ol{k}}$. 
For $\lambda \in \bX_{\bullet}(T)$ and 
$\mu \in \bX_{\bullet}(T)^+$, 
an irreducible component of 
$S_{\lambda} \cap \Gr_{\mu}$ is called 
a Mirkovi\'{c}--Vilonen cycle after \cite{MiViGeoLd}. 
Let $\bMV_{\mu}(\lambda)$ be the 
set of the Mirkovi\'{c}--Vilonen cycles in 
$S_{\lambda} \cap \Gr_{\mu}$. 
We sometimes write 
$(S_{\lambda} \cap \Gr_{\mu})^{\bfb}$ 
instead of 
$\bfb \in \bMV_{\mu}(\lambda)$ 
for the Mirkovi\'{c}--Vilonen cycle. 

Let $(\wh{G},\wh{B},\wh{T})$ 
be the Langlands dual 
over $\ol{\bQ}_{\ell}$ of $(G,B,T)$. 
For $\mu \in \bX_{\bullet}(T)^+=\bX^{\bullet}(\wh{T})^+$, 
let $V_{\mu}$ denote the irreducible algebraic representation of 
$\wh{G}$ of highest weight $\mu$. 
For an algebraic representation $V$ of $\wh{G}$ 
and $\lambda \in \bX_{\bullet}(T)=\bX^{\bullet}(\wh{T})$, 
let $V(\lambda)$ denote the $\lambda$-weight space of $V$. 
Then we have 
\begin{equation}\label{eq:MVdim}
 \lvert \bMV_{\mu}(\lambda) \rvert = \dim V_{\mu} (\lambda) 
\end{equation}
by \cite[Proposition 5.4.2]{GHKRDimDL} and 
\cite[Corollary 2.8]{ZhuAffGmix}. 

For $\nu, \mu \in \bX_{\bullet}(T)^+$ and 
$\lambda \in \bX_{\bullet}(T)$ such that 
$\nu+\lambda \in \bX_{\bullet}(T)^+$, 
there is an injective map 
\begin{equation*}
 i_{\nu}^{\bMV} \colon \bS_{(\nu,\mu)|\nu+\lambda} \to 
 \bMV_{\mu} (\lambda) 
\end{equation*}
constructed by \cite[Lemma 3.2.7]{XiZhCycSh}. 

\section{Equidimensionality of Satake cycles}\label{sec:EqScyc}
Let $\mu_{\bullet} =(\mu_1, \ldots, \mu_n) \in (\bX_{\bullet}(T)^+)^n$ and $\lambda \in \bX_{\bullet}(T)^+$. 
\begin{lem}\label{lem:Ztri}
The morphism 
$m_{\mu_{\bullet}} \colon \Gr_{\mu_{\bullet}} \to \Gr_{\ol{k}}$ 
is Zariski-locally trivial over $\mathring{\Gr}_{\lambda}$ in the sense that for any point $y$ of $\mathring{\Gr}_{\lambda}$ there is a Zariski open subspace $V \subset \mathring{\Gr}_{\lambda}$ with $y \in V$ and $\overline{k}$-scheme $Y$ such that $m_{\mu_{\bullet}}^{-1}(V) \to V$ is identified with the projection $Y \times_{\overline{k}} V \to V$. 
\end{lem}
\begin{proof}
Taking the base change to 
an unramified extension of $\cO_F$, 
we may assume that $G$ is split by \eqref{eq:Grbc}. 
As in the proof of \cite[Lemma 2.1]{HaiEquconv}, 
it suffices to show that 
\[
 \rL^+G \to \rL^+G/(\rL^+G \cap \varpi^{\lambda} \rL^+G \varpi^{-\lambda}) 
\]
has a section Zariski-locally. 
Since 
$\rL^+U/(\rL^+U \cap \varpi^{\lambda} \rL^+U \varpi^{-\lambda})$ 
is an open subscheme of 
$\rL^+G/(\rL^+G \cap \varpi^{\lambda} \rL^+G \varpi^{-\lambda})$, 
it suffices to show that 
\[
 \rL^+U \to \rL^+U/(\rL^+U \cap \varpi^{\lambda} \rL^+U \varpi^{-\lambda}) 
\]
has a section. 
We fix an identification $\Ga \simeq U_{\alpha,\cO_F}$ 
for a positive root $\alpha$. 
For a positive root $\alpha$, 
let 
$\rL^+_{< \langle \alpha, \lambda \rangle} U_{\alpha,\cO_F}$ 
be the closed subscheme of $\rL^+U_{\alpha,\cO_F}$ 
defined by the condition 
$x_i=0$ for $i \geq \langle \alpha, \lambda \rangle$ 
for a point 
$\sum_{i=0}^{\infty} \varpi^i [x_i]$ of 
$\rL^+U_{\alpha,\cO_F}$. 
Then the composition 
\[
 \prod_{\alpha} \rL^+_{< \langle \alpha, \lambda \rangle} U_{\alpha,\cO_F} \to \rL^+ U \to \rL^+U/(\rL^+U \cap \varpi^{\lambda} \rL^+U \varpi^{-\lambda}) 
\]
is an isomorphism. Hence we have a section. 
\end{proof}

\begin{lem}\label{lem:SGrA}
Assume that $\mu$ is a dominant minuscule cocharacter and $w \in W$. 
We have an isomorphism 
\[
 S_{w\mu} \cap \Gr_{\mu} \simeq {\rL^+U}_{\ol{k}}/((\rL^+U)_{\ol{k}} \cap \varpi^{w\mu} (\rL^+U)_{\ol{k}}  \varpi^{-w\mu}) . 
\]
In particular, 
$S_{w\mu} \cap \Gr_{\mu}$ is the perfection of an affine space 
of dimension $\langle \rho, \mu +w \mu \rangle$. 
\end{lem}
\begin{proof}
The first claim follows from 
\cite[(3.2.3)]{XiZhCycSh}. 
The second claim follows from the first one as in the proof of 
\cite[Lemma 3.2]{HaiEquconv}. 
\end{proof}

\begin{thm}\label{thm:edmin}
Assume that $\mu_i$ are minuscule. 
For a point $y$ of $\mathring{\Gr}_{\lambda}$, 
the fiber of 
$m_{\mu_{\bullet}} \colon \Gr_{\mu_{\bullet}} \to \Gr_{\ol{k}}$ 
at $y$ is equidimensional of dimension $\langle \rho,\lvert \mu_{\bullet} \rvert -\lambda \rangle$. 
\end{thm}
\begin{proof}
This is proved in the same way as \cite[Theorem 3.1]{HaiEquconv} 
using Lemma \ref{lem:Ztri} and Lemma \ref{lem:SGrA} instead of 
\cite[Lemma 2.1 and Lemma 3.2]{HaiEquconv} respectively. 
\end{proof}

\begin{prop}\label{prop:edsumm}
Assume that each $\mu_i$ is a sum of minuscule cocharacters. 
Then, for a point $y$ of $\mathring{\Gr}_{\lambda}$, 
any irreducible component of the fiber 
$m_{\mu_{\bullet}}^{-1}(y)$ 
whose generic point belongs to 
$\mathring{\Gr}_{\mu_{\bullet}}$ 
has dimension $\langle \rho,\lvert \mu_{\bullet} \rvert -\lambda \rangle$. 
\end{prop}
\begin{proof}
This follows from Theorem \ref{thm:edmin} 
in the same way as \cite[Proposition 4.1]{HaiEquconv}. 
\end{proof}

\section{Affine Deligne--Lusztig variety}\label{sec:AffDL}
Recall that $L=W_{\cO_F}(\ol{k})[\frac{1}{\varpi}]$. 
Let $b \in G(L)$ and $\mu \in \bX_{\bullet}(T)$.  
Let $\sigma$ denote the $q$-th power Frobenius element. 
We define the affine Deligne--Lusztig variety $X_{\mu}(b)$ 
by 
\[
 X_{\mu}(b) = \{ g (\rL^+G)_{\ol{k}}  \in \Gr_{\ol{k}} \mid g^{-1} b \sigma (g) \in \overline{ (\rL^+ G)_{\ol{k}} \varpi^{\mu} (\rL^+ G)_{\ol{k}} } \}. 
\]
Let $B(G)$ be the set of $\sigma$-conjugacy classes of $G(L)$. 
We define $B(G,\mu) \subset B(G)$ as in 
\cite[6.2]{KotIsoII}. 
Then $X_{\mu}(b)$ is non-empty if and only if 
$[b] \in B(G,\mu)$ by \cite[Theorem 5.1]{GasConjKR}. 

An element of $B(G)$ is called unramified if 
it is contained in the image of the natural map 
$B(T) \to B(G)$. 
Let $B(G)_{\ur}$ denote the set of unramified elements of $B(G)$. 

For $\chi \in \bX^{\bullet}(T)$, 
we put 
\[
 \ol{\chi} = \frac{1}{\lvert 
 \langle \sigma \rangle 
 \chi \rvert} \sum_{\chi' \in \langle \sigma \rangle  \chi} \chi' \in  \bX^{\bullet}(T)_{\bQ}^{\sigma} . 
\]
The natural pairing $\bX_{\bullet}(T) \times \bX^{\bullet}(T) \to \bZ$ induces a pairing 
$\langle \ , \ \rangle \colon \bX_{\bullet}(T)_{\sigma} \times \bX^{\bullet}(T)_{\bQ}^{\sigma} \to \bQ$. 
We put 
\[
 \bX_{\bullet}(T)_{\sigma}^+ = 
 \{ [\lambda] \in \bX_{\bullet}(T)_{\sigma} \mid \textrm{$\langle [\lambda] , \ol{\alpha} \rangle \geq 0$ for every  $\alpha \in \Delta$} \}. 
\]
Then we have the bijection 
\[
 \bX_{\bullet}(T)_{\sigma}^+ \simeq 
 B(G)_{\ur};\ [\lambda] \mapsto [\varpi^{\lambda}] 
\]
as in \cite[Lemma 4.2.3]{XiZhCycSh}. 

For $\tau \in \bX_{\bullet}(T)$, we write 
$X_{\mu} (\tau)$ for $X_{\mu} (\varpi^{\tau})$. 
We assume that $b=\varpi^{\tau}$ 
for $\tau \in \bX_{\bullet}(T)$ 
such that $[\tau] \in \bX_{\bullet}(T)_{\sigma}^+$. 
We can define the twisted centralizer $J_{\tau}$ over $\cO_F$ 
for $\varpi^{\tau}$ 
as in \cite[4.2.13]{XiZhCycSh}. 
We note that $J_{\tau}=G$ if $[b] \in B(G)_{\ur}$ is basic.

We assume \cite[Hypothesis 4.4.1]{XiZhCycSh} for $J_{\tau}$. 
Further, we assume that $Z_G$ is connected. 

Let $\lambda \in  \bX_{\bullet}(T)$ such that 
$[\lambda]=[\tau] \in \bX_{\bullet}(T)_{\sigma}^+$. 
We take $\delta_{\lambda} \in \bX_{\bullet}(T)$ 
such that 
$\lambda =\tau +\delta_{\lambda} -\sigma (\delta_{\lambda})$. 
Let $\bfb \in \bMV_{\mu}(\lambda)$. 
Consider the condition $\textrm{C}_{\bfb}$ for $\nu \in \bX_{\bullet}(T)$ that 
\[
\textrm{$\lambda +\nu -\sigma (\nu)$ is dominant and $\bfb$ is in the image of $i_{\nu}^{\bMV} \colon \bS_{(\nu,\mu)|\lambda +\nu} \to \bMV_{\mu} (\lambda)$.} 
\]
By \cite[Lemma 4.4.3]{XiZhCycSh}, we take $\nu_{\bfb} \in \bX_{\bullet}(T)$ such that $\nu_{\bfb}$ satisfies $\textrm{C}_{\bfb}$ and $\nu -\nu_{\bfb}$ is dominant for any $\nu \in \bX_{\bullet}(T)$ satisfying $\textrm{C}_{\bfb}$. Such a $\nu_{\bfb}$ is unique up to $\bX_{\bullet}(Z_G)$ by the same lemma. 
We put $\tau_{\bfb} =\lambda +\nu_{\bfb} -\sigma (\nu_{\bfb})$. 
Then we have the isomorphism 
\[
 J_{\tau} (F) \simeq J_{\tau_{\bfb}} (F);\ 
 g \mapsto \varpi^{\delta_{\lambda} +\nu_{\bfb}} g 
 \varpi^{-\delta_{\lambda} -\nu_{\bfb}}. 
\]
We consider the isomorphism 
\begin{equation}\label{eq:ttbiso}
 X_{\mu} (b) = 
 X_{\mu} (\tau) \simeq X_{\mu} (\tau_{\bfb}) ; \ 
 g \rL^+ G \mapsto  
 \varpi^{\delta_{\lambda} +\nu_{\bfb}} 
 g \rL^+ G . 
\end{equation}
Let $\bfa \in \bS_{(\nu_{\bfb},\mu)|\lambda +\nu_{\bfb}}$ 
be the unique element such that 
$\bfb =i_{\nu_{\bfb}}^{\bMV} (\bfa)$.

We define $X_{\mu,\nu_{\bfb}}(\tau_{\bfb})$ by 
the fiber product 
 \[
 \xymatrix{
 X_{\mu,\nu_{\bfb}}(\tau_{\bfb})  \ar[rr]^{}  \ar[d] &     &     \Gr^{0}_{(\nu_\bfb, \mu) | \tau_{\bfb} + \sigma( \nu_\bfb) }              \ar[d]^{\mathrm{pr}_1 \times m} \\ 
 \Gr_{\nu_\bfb}  \ar[rr]^-{1 \times \varpi^{\tau_{\bfb}} \sigma }   &  &  \Gr_{\nu_\bfb} \times \Gr_{\tau_{\bfb} + \sigma( \nu_\bfb ) }.} 
\]
More concretely, we have 
\[
 X_{\mu,\nu_{\bfb}}(\tau_{\bfb}) = \{ g (\rL^+G)_{\ol{k}}  \in \Gr_{\nu_{\bfb}} \mid g^{-1} \varpi^{\tau_{\bfb}} \sigma (g) \in \overline{ (\rL^+ G)_{\ol{k}} \varpi^{\mu} (\rL^+ G)_{\ol{k}} } \}.
\]
Further, we define $X_{\mu,\nu_{\bfb}}^{\bfa}(\tau_{\bfb})$ by 
the fiber product 
 \[
 \xymatrix{
 X_{\mu,\nu_{\bfb}}^{\bfa}(\tau_{\bfb})  \ar[rr]^{}  \ar[d] &     &     \Gr^{0,\bfa}_{(\nu_\bfb, \mu) | \tau_{\bfb} + \sigma( \nu_\bfb) }              \ar[d] \\ 
 X_{\mu,\nu_{\bfb}}(\tau_{\bfb})  \ar[rr]   &  &  \Gr^{0}_{(\nu_\bfb, \mu) | \tau_{\bfb} + \sigma( \nu_\bfb) }.} 
\]
Let $x_0$ denote $[1] \in J_{\tau}(F)/J_{\tau}(\cO_F)$. 
We put 
\[
 \mathring{X}_{\mu}^{\bfb,x_0}(\tau_{\bfb}) = 
 X_{\mu,\nu_{\bfb}}^{\bfa}(\tau_{\bfb}) \cap 
 \mathring{\Gr}_{\nu_{\bfb}}. 
\]
Let $X_{\mu}^{\bfb,x_0}(\tau_{\bfb})$ 
denote the closure of 
$\mathring{X}_{\mu}^{\bfb,x_0}(\tau_{\bfb})$ 
in $X_{\mu,\nu_{\bfb}}^{\bfa}(\tau_{\bfb})$. 
The scheme $X_{\mu}^{\bfb,x_0}(\tau_{\bfb})$ is irreducible of dimension $\langle \rho, \mu -\tau_{\bfb} \rangle$ by \cite[Theorem 4.4.5]{XiZhCycSh}. 

By \cite[Theorem 4.4.14]{XiZhCycSh}, 
there is a bijection between the set
\[
 \bigsqcup_{\lambda \in  \bX_{\bullet}(T),\ [\lambda]=[\tau] \in \bX_{\bullet}(T)_{\sigma}^+}  \bMV_{\mu}(\lambda) \times J_{\tau}(F)/J_{\tau}(\cO_F) 
\] 
and the set of irreducible components of $X_{\mu}(b)$ given by 
\[
 (\bfb,[g]) \mapsto X_{\mu}^{\bfb,[g]}(\tau_{\bfb}) 
 \coloneqq 
 g X_{\mu}^{\bfb,x_0}(\tau_{\bfb}), 
\]
where we regard $X_{\mu}^{\bfb,[g]}(\tau_{\bfb})$ 
as a subscheme of $X_{\mu}(b)$ by \eqref{eq:ttbiso}.

\section{Unitary group}\label{sec:Uni}

\subsection{Setting}
Let $F_2$ be the quadratic unramified extension of $F$. Let $\cO_{F_2}$ denote the ring of integers of $F_2$. 
Let $\varpi$ be a uniformizer of $F$. 
We put $\Lambda=\cO_{F_2}^n$ equipped with the hermitian form 
\begin{equation}\label{eq:HermOpair}
 \cO_{F_2}^n \times \cO_{F_2}^n \to \cO_{F_2};\ 
 ((a_i)_{1 \leq i \leq n},(a_i')_{1 \leq i \leq n}) 
 \mapsto \sum_{i=1}^n \sigma (a_i) a_{n+1-i}' . 
\end{equation}
We put $G=\GU (\Lambda)$. 
By taking the first factor of the decomposition 
\[
 \cO_{F_2} \otimes_{\cO_F} \cO_{F_2} \simeq \cO_{F_2} \times \cO_{F_2};\ 
 a \otimes b \mapsto (ab,a\sigma(b)) , 
\]
we have an isomorphism 
\begin{equation}\label{eq:punisp}
 G_{\cO_{F_2}} \simeq \GL_n \times \Gm . 
\end{equation}
We put $V=\Lambda \otimes_{\cO_F} F$. 
Let $\widehat{G} = \GL_n \times \Gm$ denote the dual group over $\ol{\bQ}_{\ell}$ 
with a maximal torus $\widehat{T}$ and a Borel subgroup $\widehat{B}$, 
which are the diagonal torus and the upper triangular subgroup on the $\GL_n$-component. 
For $\mu \in \bX_{\bullet}(T)^+$, let 
$\mu^* \in \bX_{\bullet}(T)^+$ be the element such that 
$V_{\mu^*}=V_{\mu}^*$. 

For an index $i \in \{1, \dots, n \}$, we will use the notation $i^\vee = n + 1 - i$. The group $\mathbb{X}^{\bullet}( \widehat{T})$ has a basis $\{ \varepsilon_i \}_{i=0}^n$, where 
$\varepsilon_0$ is the projection to the $\Gm$-component and 
$\varepsilon_i$ is the character of $\widehat{T}$ given by evaluating the $(i,i)$ entry for $i \geq 1$. In the following, all cocharacters of $T$ (equivalently, characters of $\widehat{T}$) will be written according to this basis. 
We have 
$\sigma (\varepsilon_0)=\sum_{i=0}^n \varepsilon_i$ and 
$\sigma (\varepsilon_i)=-\varepsilon_{i^{\vee}}$ for $1 \leq i \leq n$. 
For $\mu =\sum_{i=0}^n m_i \varepsilon_i \in \bX_{\bullet}(T)^+$, we have \[
 \mu^* = -m_0 \varepsilon_0 - \sum_{i=1}^n m_{n+1-i} \varepsilon_i \in \bX_{\bullet}(T)^+. 
\]
\subsection{Satake cycle}
Let $\mu =\varepsilon_0 + \varepsilon_1 +\varepsilon_2 \in \bX_{\bullet}(T)$. 
We put $r=[n/2]$. 
We put 
\[
 \nu_i=\varepsilon_1 + \cdots + \varepsilon_{i-1} 
 -\varepsilon_{i^{\vee}} - \cdots - \varepsilon_{1^{\vee}}, \quad 
 \tau_i=\varepsilon_0 
\]
for $1 \leq i \leq [(n-1)/2]$, and 
\[
 \nu_r=\varepsilon_1 + \cdots + \varepsilon_{r-1} , 
 \quad 
 \tau_r=\varepsilon_0 +\varepsilon_1 + \cdots + \varepsilon_n 
\]
if $n$ is even. 
We put $\lambda_i=-\varepsilon_0 -\varepsilon_i -\varepsilon_{i^{\vee}}$ for $1 \leq i \leq r$. 

\begin{lem}\label{lem:MVnz}
For $\lambda \in -\varepsilon_0 + (1 - \sigma)\bX_{\bullet}(T)$, 
we have 
$\bMV_{\mu^*}(\lambda) \neq \emptyset$ if and only if 
$\lambda \in \{ \lambda_1, \ldots , \lambda_r \}$. 
Further, $\bMV_{\mu^*}(\lambda_i)$ is a singleton for $1 \leq i \leq r$. 
\end{lem}
\begin{proof}
For $\lambda \in \bX_{\bullet}(T)$, 
we have $\dim V_{\mu^*} (\lambda) \leq 1$, and 
$V_{\mu^*} (\lambda)$ is nonzero if and only if 
$\lambda=-\varepsilon_0 -\varepsilon_i - \varepsilon_j$ for some $1 \leq i < j \leq n$. 
If $-\varepsilon_0-\varepsilon_i - \varepsilon_j \in -\varepsilon_0 +(1 - \sigma)\bX_{\bullet}(T)$ 
for some $1 \leq i < j \leq n$, we must have $j=i^{\vee}$. 
Hence the claim follows from \eqref{eq:MVdim}. 
\end{proof}

Let $1 \leq i \leq r$. 
Note that $\tau_i^* +\sigma (\nu_i^*)= \nu_i^* +\lambda_i =\nu_i+\tau_i^*$. 
Let $\bfb_i$ be the unique element of 
$\bMV_{\mu^*}(\lambda_i)$. 
There is $\bfa_i \in  \mathbb{S}_{(\nu_i^*,\mu^*)|\nu_i+\tau_i^*}$ such that 
$i_{\nu_i^*}^{\mathbb{MV}} (\bfa_i)=\bfb_i$. 
Since $i_{\nu_i^*}^{\mathbb{MV}}$ is injective by \cite[Lemma 3.2.7]{XiZhCycSh}, the set $\mathbb{S}_{(\nu_i^*,\mu^*)|\nu_i+\tau_i^*}$ is also a singleton.

We study the Satake cycle 
$\Gr^{0,\bfa_i}_{(\nu^*_i, \mu^*) | \nu_i+\tau_i^* }$. 
\begin{lem}\label{lem:Gracl}
\begin{enumerate}
\item\label{en:ringGr}
The scheme 
$\mathring{\Gr}^{0}_{(\nu^*_i, \mu^*) | \nu_i +\tau_i^* }$ 
is irreducible. 
\item\label{en:Gra} 
We have 
\[
 \Gr^{0,\bfa_i}_{(\nu^*_i, \mu^*) | \nu_i +\tau_i^*}=\ol{\mathring{\Gr}^{0}_{(\nu^*_i, \mu^*) | \nu_i +\tau_i^*}}. 
\]
\end{enumerate}
\end{lem}
\begin{proof}
Since $\mathbb{S}_{(\nu_i^*,\mu^*)|\nu_i+\tau_i^*}$ is a singleton, by Lemma \ref{lem:SatCyCl}, 
we know that there is only one irreducible component of $\mathring{\Gr}^{0}_{(\nu^*_i, \mu^*) | \nu_i +\tau_i^* }$ whose dimension is equal to the dimension of $\mathring{\Gr}^{0}_{(\nu^*_i, \mu^*) | \nu_i +\tau_i^* }$. For the claim \ref{en:ringGr}, 
it remains to show that $\mathring{\Gr}^{0}_{(\nu^*_i, \mu^*) | \nu_i +\tau_i^* }$ is equidimensional. 
By the definition, 
$\mathring{\Gr}^{0}_{(\nu^*_i, \mu^*) | \nu_i +\tau_i^* }$ is equal to the inverse image of $\mathring{\Gr}_{\nu_i +\tau_i^* }$ 
under the convolution morphism 
\[
 m_{(\nu^*_i, \mu^*)} \colon 
 \Gr_{(\nu^*_i, \mu^*)} \to \Gr. 
\]
Therefore the equidimensionality of $\mathring{\Gr}^{0}_{(\nu^*_i, \mu^*) | \nu_i +\tau_i^* }$ follows from 
Lemma \ref{lem:Ztri} and
Proposition \ref{prop:edsumm}.  
The claim \ref{en:Gra} follows from \ref{en:ringGr}. 
\end{proof}

We do not use the following lemma in the sequel, 
but it shows that a study of 
intersections of irreducible components of affine Deligne--Lusztig varieties is 
more subtle than intersections of Satake cycles. 

\begin{lem}\label{lem:Scon}
Assume that $n \geq 5$. 
\begin{enumerate}
\item\label{en:tran12}
The actions of $\rL^+G$ on 
$\mathring{\Gr}^{0}_{(\nu^*_2, \mu^*) | \nu_1 +\tau_1^*}$ and 
$\Gr^{0}_{(\nu^*_1, \mu^*) | \nu_1 +\tau_1^*}$ are 
transitive. 
\item\label{en:cont12} 
The Satake cycle 
$\Gr^{0,\bfa_2}_{(\nu^*_2, \mu^*) | \nu_2 +\tau_2^*}$ contains 
$\Gr^{0,\bfa_1}_{(\nu^*_1, \mu^*) | \nu_1 +\tau_1^*}$. 
\end{enumerate}
\end{lem}
\begin{proof}
We show \ref{en:tran12}. 
It suffices to show that the number of 
the orbits under the action of $\rL^+G$ on 
$\Gr^{0}_{(\nu^*_2, \mu^*) | \nu_1 +\tau_1^*}$ is $2$. 
Let $(\rL^+G)_{\nu_1 +\tau_1^*}$ be the stabilizer of 
$[\varpi^{\nu_1+\tau_1^*}] \in \Gr_{\nu_1+\tau_1^*}$ in $\rL^+G$. 
Since the action of $\rL^+G$ on $\Gr_{\nu_1 +\tau_1^*}$ is transitive, 
it suffices to show that the number of the orbits in 
$m_{(\nu^*_2, \mu^*) | \nu_1 +\tau_1^* }^{-1}([\varpi^{\nu_1 +\tau_1^*}])$ 
under the action of $(\rL^+G)_{\nu_1 +\tau_1^*}$ 
is $2$. 
These orbits are in a bijection with $(P_{\nu_1 +\tau_1^*,\cO_L})_{\ol{k}} \backslash G_{\ol{k}} / (P_{\mu,\cO_L})_{\ol{k}}$. 
Hence the number of the orbits is $2$. 

We show \ref{en:cont12}. 
By Lemma \ref{lem:SatSurj}, the natural morphism 
$\Gr^{0,\bfa_2}_{(\nu^*_2, \mu^*) | \nu_2+\tau_2^*} \to \Gr_{\nu_2+\tau_2^*}$ is surjective. 
Hence the intersection of $\Gr^{0,\bfa_2}_{(\nu^*_2, \mu^*) | \nu_2+\tau_2^*}$ and $\Gr^{0}_{(\nu^*_2, \mu^*) | \nu_1 +\tau_1^*}$ is not empty. 

If the intersection of $\Gr^{0,\bfa_2}_{(\nu^*_2, \mu^*) | \nu_2+\tau_2^*}$ and $\mathring{\Gr}^{0}_{(\nu^*_2, \mu^*) | \nu_1 +\tau_1^*}$ is not empty, 
then $\Gr^{0,\bfa_2}_{(\nu^*_2, \mu^*) | \nu_2 +\tau_2^*}$ contains  $\mathring{\Gr}^{0}_{(\nu^*_2, \mu^*) | \nu_1 +\tau_1^*}$ 
because $\rL^+G$ acts transitively on $\mathring{\Gr}^{0}_{(\nu^*_2, \mu^*) | \nu_1 +\tau_1^*}$ and $\mathring{\Gr}^{0}_{(\nu^*_2, \mu^*) | \nu_2 +\tau_2^*}$ is stable under the action of $\rL^+G$. Then $\Gr^{0,\bfa_2}_{(\nu^*_2, \mu^*) | \nu_2 +\tau_2^*}$ contains $\Gr^{0}_{(\nu^*_1, \mu^*) | \nu_1 +\tau_1^*}$, since  $\mathring{\Gr}^{0}_{(\nu^*_2, \mu^*) | \nu_1 +\tau_1^*}$ is dense in $\Gr^{0}_{(\nu^*_2, \mu^*) | \nu_1 +\tau_1^*}$. 

If the intersection of $\Gr^{0,\bfa_2}_{(\nu^*_2, \mu^*) | \nu_2 +\tau_2^*}$ and $\mathring{\Gr}^{0}_{(\nu^*_2, \mu^*) | \nu_1 +\tau_1^*}$ is empty, 
the intersection of $\Gr^{0,\bfa_2}_{(\nu^*_2, \mu^*) | \nu_2 +\tau_2^*}$ and $\Gr^{0}_{(\nu^*_1, \mu^*) | \nu_1 +\tau_1^*}$ is not empty. Then 
$\Gr^{0,\bfa_2}_{(\nu^*_2, \mu^*) | \nu_2 +\tau_2^*}$ contains $\Gr^{0}_{(\nu^*_1, \mu^*) | \nu_1 +\tau_1^*}$ because 
$\rL^+G$ acts transitively on $\Gr^{0}_{(\nu^*_1, \mu^*) | \nu_1 +\tau_1^*}$ and $\mathring{\Gr}^{0}_{(\nu^*_2, \mu^*) | \nu_2 +\tau_2^*}$ is stable under the action of $\rL^+G$. 
\end{proof}

\begin{rem}\label{rem:quesXZ}
Lemma \ref{lem:Scon} \ref{en:cont12} shows that 
$X^{\bfa_2}_{\mu^*,\nu^*_2}(\tau_2^*)$ is not irreducible. 
This answers a question in \cite[Remark 4.4.6 (3)]{XiZhCycSh}. 
\end{rem}

\section{Irreducible Components}\label{sec:Irr}

We note that 
$[\varpi^{-\varepsilon_0}] \in B(G , \mu^*)$ is the basic class. 
Let $X_{\mu^*}^{\bfb_i,x_0}(\tau_i^*)$ be 
the closure in $X_{\mu^*}(\tau_i^*)$ of 
$\mathring{X}_{\mu^*}^{\bfb_i,x_0}(\tau_i^*)$ fitting in the cartesian diagram 
 \[
\xymatrix{
\mathring{X}_{\mu^*}^{\bfb_i, x_0}(\tau_i^*)  \ar[rr]^{}  \ar[d] &     &     \mathring{\Gr}^{0,\bfa_i}_{(\nu^*_i, \mu^*) | \nu_i+\tau_i^* }              \ar[d]^{\mathrm{pr}_1 \times m} \\ 
\mathring{\Gr}_{\nu^*_i}  \ar[rr]^-{1 \times \varpi^{\tau^*_i} \sigma }   &  &  \mathring{\Gr}_{\nu^*_i} \times \mathring{\Gr}_{\nu_i+\tau_i^*}  . \\
}\]
By results in \S \ref{sec:AffDL} and Lemma \ref{lem:MVnz}, we obtain the following proposition: 
\begin{prop}\label{prop:irralluni}
The number of the $G(F)$-orbits of the irreducible components of 
$X_{\mu^*}(\varepsilon_0^*)$ is $r$. 
Representatives of $r$ orbits are 
given by $X_{\mu^*}^{\bfb_i,x_0}(\tau_i^*)$ for $1 \leq i \leq r$. 
The $G(F)$-orbit of $X_{\mu^*}^{\bfb_i,x_0}(\tau_i^*)$ is parametrized by 
$G(F)/G(\cO_F)$. 
The dimension of $X_{\mu^*}^{\bfb_i,x_0}(\tau_i^*)$ 
is $\langle \rho, \mu^* - \tau_i^* \rangle=n-2$. 
\end{prop}

If $i = 1$, the above construction defines a Deligne--Lusztig variety. 
If $i = 2$ and $n \geq 5$, 
this defines a variety that is not a Deligne--Lusztig variety.

By \eqref{eq:Grbc} and \eqref{eq:punisp}, 
we have an isomorphism 
\begin{equation}\label{eq:GrGsp}
 \Gr_G \otimes_{\bF_q} \bF_{q^2} \simeq \Gr_{\GL_n \times \Gm}.     
\end{equation}
We put $\cE^0=\cO_{F_2}^n$ and $\cL^0=\cO_{F_2}$ and view them as trivial vector bundles 
on $D_{\bF_{q^2}}$. We have an isomorphism 
\begin{equation}\label{eq:Fdualid}
 \cE^0 \simeq \oF ((\cE^0)^{\vee}) 
\end{equation}
given by \eqref{eq:HermOpair}. 
For any perfect $\bF_{q^2}$-algebra $R$, 
\begin{equation}\label{eq:GrnGm}
 \Gr_{\GL_n \times \Gm} (R) = \left\{ 
 (\cE, \cL, \beta, \beta') \relmiddle{|} 
 \begin{tabular}{ l }
  $\cE$ is a vector bundle on $D_R$ of rank $n$, \\ 
  $\cL$ is a line bundle on $D_R$, \\ 
  $\beta \colon \cE|_{D^*_R} \simeq \cE^0|_{D^*_R}$ and \\ 
  $\beta' \colon \cL|_{D^*_R} \simeq \cL^0|_{D^*_R}$ are trivializations. 
  \end{tabular} 
\right\} 
\end{equation}
by \eqref{eq:Grmod}. 
Under the identification by \eqref{eq:GrGsp}, 
the Frobenius endomorphism of 
$\Gr_G \otimes_{\bF_q} \bF_{q^2}$ sends $(\cE, \cL, \beta, \beta')$ 
in \eqref{eq:GrnGm} to 
\[
 (\oF (\cE^{\vee}) \otimes \cL, \cL, \oF (\beta^{\vee})^{-1} \otimes \beta', \beta')
\]
where 
\[
 \oF (\beta^{\vee})^{-1} \otimes \beta' \colon 
 (\oF (\cE^{\vee}) \otimes \cL)|_{D^*_R} \simeq (\oF ((\cE^0)^{\vee}) \otimes \cL^0)|_{D^*_R} 
 \simeq \cE^0 |_{D^*_R}  
\]
using \eqref{eq:Fdualid} at the last isomorphism. 
We regard $\Gr_{\GL_n}$ as an open and closed sub-ind-scheme of $\Gr_{\GL_n \times \Gm}$ 
by 
\[
 (\cE,\beta) \mapsto (\cE, \cL^0, \beta, \id) . 
\]
If $\lambda \in \bX_{\bullet}(T)$ is trivial on $\Gm$-component 
under the identification \eqref{eq:punisp}, 
then we view $\Gr_{G,\lambda}$ as a subscheme of 
$\Gr_{\GL_n} \subset \Gr_{\GL_n \times \Gm}$ under the identification 
\eqref{eq:GrGsp}. 
Under the identification by \eqref{eq:GrGsp}, 
the Frobenius endomorphism of 
$\Gr_G \otimes_{\bF_q} \bF_{q^2}$ becomes 
\[
 (\cE,\beta) \mapsto (\oF (\cE^{\vee}) , \oF (\beta^{\vee})^{-1}) 
\]
in $\Gr_{\GL_n}$. 
When we compare the positions two vector bundles on $D_{R}$ equipped with trivializations over $D_{R}^*$, they are compared through the trivializations. 
We put $\mu_{\GL}=\varepsilon_1 +\varepsilon_2$.

\subsection{Component for $\nu_1$}
The following proposition describes an analogue of a component studied in \cite{VoWeGUII}. 

\begin{prop}\label{prop:nu1}
The irreducible component 
$X_{\mu^*}^{\bfb_1,x_0}(\tau_1^*)$ is parametrized by $\cE \dashrightarrow \cE^0$ bounded by 
$\nu_1^*$ such that 
$\varpi \oF (\cE^{\vee}) \subset \cE$. 
In particular, it is isomorphic to 
$X_{I_{n}^{1}}([1])^{\pf}$. 
\end{prop}
\begin{proof}
We have 
\[
 \Gr^{0,\bfa_1}_{(\nu^*_1, \mu^*) | \nu_1 +\tau_1^*}
 =\Gr^{0}_{(\nu^*_1, \mu^*) | \nu_1 +\tau_1^*}
 =\mathring{\Gr}^{0}_{(\nu^*_1, \mu^*) | \nu_1 +\tau_1^*} 
\] 
since $\nu_1$ is minuscule. 
Hence we have 
$X_{\mu^*}^{\bfb_1,x_0}(\tau_1^*)=\mathring{X}_{\mu^*}^{\bfb_1,x_0}(\tau_1^*)=X_{\mu^*,\nu_1^*}(\tau_1^*)$. 
By the definition, 
$X_{\mu^*,\nu_1^*}(\tau_1^*)$ 
is parametrized by 
$\cE \dashrightarrow \cE^0$ bounded by 
$\nu_1^*$ such that 
$\oF (\cE^{\vee}) \dashrightarrow \cE$ 
is bounded by $\mu_{\GL}^*$. 
We note that the last condition is equivalent to that 
$\cE \dashrightarrow \oF (\cE^{\vee})$ 
is bounded by $\mu_{\GL}$. 
If $\cE \dashrightarrow \cE^0$ is bounded by 
$\nu_1^*=\varepsilon_1$, then $\cE^0 \dashrightarrow \oF (\cE^{\vee})$ is also bounded by 
$\varepsilon_1$. 
Therefore, when $\cE \dashrightarrow \cE^0$ is bounded by 
$\nu_1^*$, 
the condition that 
$\cE \dashrightarrow \cE^0 \dashrightarrow \oF (\cE^{\vee})$ 
is bounded by $\mu_{\GL}$ is equivalent to 
$\cE \subset \oF (\cE^{\vee}) \subset \varpi^{-1}\cE$. 
This is further equivalent to 
$\varpi \oF (\cE^{\vee}) \subset \cE$, since we already have 
$\cE \subset \cE^0 \subset \oF (\cE^{\vee})$. 

The last isomorphism in the claim 
is given by sending 
$\cE$ to 
$\oF (\cE^{\vee}) /\cE^0 \subset \frac{1}{\varpi} \cE^0/\cE^0$, where we have $\cE^0 \subset \oF (\cE^{\vee})$ by $\cE \subset \cE^0$ and \eqref{eq:Fdualid}.  
\end{proof}

By Proposition \ref{prop:nu1} and Example \ref{eq:Fermat}, $X_{\mu^*}^{\bfb_1,x_0}(\tau_1^*)$ 
is isomorphic to 
the perfection of the Fermat hypersurface 
defined by 
\[
\sum_{i=1}^n x_i x_{n+1-i}^q=0 
\]
in $\bP^{n-1}$. 

\subsection{Components for $\nu_r$ when $n$ is even.}
The following proposition describes a generalization of a component studied in \cite{HoPaGU22}. 

\begin{prop}\label{prop:nur}
Assume that $n$ is even. The irreducible component 
$X_{\mu^*}^{\bfb_r,x_0}(\tau_r^*)$ is 
parametrized by 
$\cE \dashrightarrow \cE^0$ bounded by 
$\nu_r^*$ such that 
$\varpi \cE \subset \oF (\cE^{\vee})$. 
In particular, it is isomorphic to 
$X_{I_{n}^{r-1}}([1])^{\mathrm{pf}}$. 
\end{prop}
\begin{proof}
We have 
\[
 \Gr^{0,\bfa_r}_{(\nu^*_r, \mu^*) | \nu_r+\tau^*_r }
 =\Gr^{0}_{(\nu^*_r, \mu^*) | \nu_r+\tau^*_r }
 =\mathring{\Gr}^{0}_{(\nu^*_r, \mu^*) | \nu_r+\tau^*_r } 
\] 
since $\nu_r^*$ and $\nu_r+\tau^*_r$ are minuscule. Hence we have 
$X_{\mu^*}^{\bfb_r,x_0}(\tau_r^*)=\mathring{X}_{\mu^*}^{\bfb_r,x_0}(\tau_r^*)=X_{\mu^*,\nu_r^*}(\tau_r^*)$. 

By the definition, 
$X_{\mu^*,\nu_r^*}(\tau_r^*)$ 
is parametrized by 
$\cE \dashrightarrow \cE^0$ bounded by 
$\nu_r^*$ such that 
$\frac{1}{\varpi} \oF (\cE^{\vee}) \dashrightarrow \cE$ 
is bounded by $\mu_{\GL}^*$, because the $\GL_n$-component of $\tau_r^*(\varpi)$ is the scalar matrix $\varpi^{-1}$. 
The condition that 
$\frac{1}{\varpi} \oF (\cE^{\vee}) \dashrightarrow \cE$ 
is bounded by $\mu_{\GL}^*$ is equivalent to 
$\varpi \cE \subset \oF (\cE^{\vee})$. 

The last isomorphism in the claim 
is given by sending 
$\cE$ to 
$\cE/\cE^0 \subset \frac{1}{\varpi} \cE^0/\cE^0$. 
\end{proof}

\begin{eg}
Assume that $n=4$. Then 
$X_{\mu^*}^{\bfb_r,x_0}(\tau_r^*)$ is isomorphic to 
the perfection of the Fermat hypersurface 
defined by 
\[
 x_1 x_4^q + x_2 x_3^q + x_3 x_2^q +x_4 x_1^q =0 
\]
in $\bP^3$. 
This is a component which appears in \cite[p.1689]{HoPaGU22}. 
\end{eg}

\subsection{Non-minuscule case}\label{sebsec:nonmin}

Let $2 \leq i \leq [(n-1)/2]$. 
We put 
\[
 \nu_{i,+} =\varepsilon_1 + \cdots +\varepsilon_{i-1}, \quad 
 \nu_{i,-} =-\varepsilon_{i^{\vee}}-\cdots-\varepsilon_{1^{\vee}}. 
\]
We put $\xi_i = \varepsilon_1 + \cdots + \varepsilon_{2i-1}$. 
For a description of $\cE$ parametrizing $X_{\mu^*}^{\bfb_i,x_0}(\tau_i^*)$, 
we introduce an auxiliary space parametrizing modifications $\cE_- \dashrightarrow \cE^0$ and $\cE^0 \dashrightarrow \cE_+$ bounded by minuscule cocharacters such that $\cE_- \dashrightarrow \cE$ and $\cE \dashrightarrow \cE_+$ are bounded by minuscule cocharacters. 
Let $(\Gr_{\nu_{i,+}^*} \times \Gr_{\nu_{i,-}^*})_{\xi_i}$ 
be the subspace of 
$\Gr_{\nu_{i,+}^*} \times \Gr_{\nu_{i,-}^*}$ 
defined by the condition that 
\[
 \cE_- \stackrel{\beta_-}{\dashrightarrow} \cE^0 \stackrel{\beta_+^{-1}}{\dashrightarrow} \cE_+ 
\]
is bounded by $\xi_i$ 
for a point 
$(\cE_+ \stackrel{\beta_+}{\dashrightarrow} \cE^0, 
 \cE_- \stackrel{\beta_-}{\dashrightarrow} \cE^0)$ 
of 
$\Gr_{\nu_{i,+}^*} \times \Gr_{\nu_{i,-}^*}$. 
Let 
\[
 (\Gr_{(\nu_{i,+}^*,\nu_{i,-}^*)} \times_{\Gr_{\nu_i^*}} \Gr_{(\nu_{i,-}^*,\nu_{i,+}^*)})_{\xi_i}
 \] 
be the subspace of 
$\Gr_{(\nu_{i,+}^*,\nu_{i,-}^*)} \times_{\Gr_{\nu_i^*}} \Gr_{(\nu_{i,-}^*,\nu_{i,+}^*)}$ 
defined by the condition that 
\[
 \cE_- \stackrel{\beta_-}{\dashrightarrow} \cE^0 \stackrel{\beta_+^{-1}}{\dashrightarrow} \cE_+ 
\]
is bounded by $\xi_i$ for a point
\[
 (\cE \stackrel{\beta_+'}{\dashrightarrow}
 \cE_+ \stackrel{\beta_+}{\dashrightarrow} \cE^0, \cE \stackrel{\beta_-'}{\dashrightarrow}
 \cE_- \stackrel{\beta_-}{\dashrightarrow} \cE^0)
\]
of $\Gr_{(\nu_{i,+}^*,\nu_{i,-}^*)} \times_{\Gr_{\nu_i^*}} \Gr_{(\nu_{i,-}^*,\nu_{i,+}^*)}$. 
We have a natural morphism 
\[
 p_1 \colon 
 (\Gr_{(\nu_{i,+}^*,\nu_{i,-}^*)} \times_{\Gr_{\nu_i^*}} \Gr_{(\nu_{i,-}^*,\nu_{i,+}^*)})_{\xi_i} \to 
 (\Gr_{\nu_{i,+}^*} \times \Gr_{\nu_{i,-}^*})_{\xi_i}. 
\]

Let $\sV_i$ be the vector bundle of rank $2i-1$ over 
$(\Gr_{\nu_{i,+}^*} \times \Gr_{\nu_{i,-}^*})_{\xi_i}$ 
defined by $\cE_+ /\cE_-$, where 
$(\cE_+, \cE_-)$ is a point of 
$(\Gr_{\nu_{i,+}^*} \times \Gr_{\nu_{i,-}^*})_{\xi_i}$. 
We put $\sG_i=\Aut (\sV_i)$. 
Let 
$t_{i,\bZ} \in \mathrm{Oc}(\mathrm{Dyn}(\GL_{2i-1,\bZ}))(\bZ)$ 
be the image under 
\[
  \mathbf{t}(\bZ) \colon \Par (\GL_{2i-1,\bZ})(\bZ) \to \mathrm{Oc}(\mathrm{Dyn}(\GL_{2i-1,\bZ}))(\bZ)  
\]
of the parabolic subgroup of $\GL_{2i-1,\bZ}$ 
defined as the stabilizer of 
$\bZ^{i-1} \subset \bZ^{i-1} \oplus \bZ^i = \bZ^{2i-1}$. 
Let 
\begin{equation}\label{eq:tidef}
 t_i \in \mathrm{Oc}(\mathrm{Dyn}(\sG_i))((\Gr_{\nu_{i,+}^*} \times \Gr_{\nu_{i,-}^*})_{\xi_i})
\end{equation}
be the element determined from $t_{i,\bZ}$ by 
Remark \ref{rem:vecig} and Lemma \ref{lem:igiso} \ref{en:Wphi}. 

We define a morphism 
\[
 \Psi \colon (\Gr_{(\nu_{i,+}^*,\nu_{i,-}^*)} \times_{\Gr_{\nu_i^*}}  \Gr_{(\nu_{i,-}^*,\nu_{i,+}^*)})_{\xi_i} \to 
  \Par_{t_i} (\sG_i) 
\]
by sending 
\[
 (\cE \stackrel{\beta_+'}{\dashrightarrow}
 \cE_+ \stackrel{\beta_+}{\dashrightarrow} \cE^0, \cE \stackrel{\beta_-'}{\dashrightarrow}
 \cE_- \stackrel{\beta_-}{\dashrightarrow} \cE^0)
\]
to the stabilizer of 
$\cE/\cE_- \subset \cE_+ / \cE_-$, 
where we have the inclusions $\cE_- \subset \cE \subset \cE_+$ 
because $\beta_+'$ and $\beta_-'$ are bounded by $\nu_{i,-}^*$ and $\nu_{i,+}^*$ respectively. 
Then $\Psi$ is an isomorphism. 
Note that a natural morphism 
\[
 p_0 \colon 
 (\Gr_{(\nu_{i,+}^*,\nu_{i,-}^*)} \times_{\Gr_{\nu_i^*}} \Gr_{(\nu_{i,-}^*,\nu_{i,+}^*)})_{\xi_i} \to \Gr_{\nu_i^*} 
\]
is an isomorphism over $\mathring{\Gr}_{\nu_i^*}$.

Recall that 
$X_{\mu^*,\nu_i^*}^{\bfa_i}(\tau_i^*)$ 
and $X_{\mu^*,\nu_i^*}(\tau_i^*)$ 
are closed subspaces of 
$\Gr_{\nu_i^*}$. 
The condition for the subspace 
\[
 p_0^{-1} (X_{\mu^*,\nu_i^*}(\tau_i^*)) \subset (\Gr_{(\nu_{i,+}^*,\nu_{i,-}^*)} \times_{\Gr_{\nu_i^*}} \Gr_{(\nu_{i,-}^*,\nu_{i,+}^*)})_{\xi_i} 
\] 
is that 
$\cE \subset \oF (\cE^{\vee}) \subset \frac{1}{\varpi} \cE$. 

For a point $(\cE,\cE_+, \cE_-)$ of 
$(\Gr_{(\nu_{i,+}^*,\nu_{i,-}^*)} \times_{\Gr_{\nu_i^*}} \Gr_{(\nu_{i,-}^*,\nu_{i,+}^*)})_{\xi_i}$, 
we put $\sW=\cE/\cE_- \subset \cE_+ /\cE_-$, 
which is a subvector bundle of rank $i-1$ since $\cE \dashrightarrow
 \cE_-$ is bounded by $\nu_{i,+}^*$. 
Let $\sW^{\perp} \subset \cE_-^{\vee} /\cE_+^{\vee}$ be the annihilator of $\sW$. 
Then we have $\sW^{\perp}=\cE^{\vee}/\cE_+^{\vee}$. 

Let $Y_i$ be the closed subscheme of 
$(\Gr_{\nu_{i,+}^*} \times \Gr_{\nu_{i,-}^*})_{\xi_i}$ 
defined by the conditions 
\begin{enumerate}
\item 
$\cE_+ \subset \oF (\cE_-^{\vee})$, 
\item 
$\cE_- \subset \oF (\cE_+^{\vee})$, 
\item
$\varpi \oF (\cE_-^{\vee}) \subset \cE_-$. 
\end{enumerate}
Then we have $Y_i=X_{I_{n}^{i-1,n-i}}([1])^{\mathrm{pf}}$ under the identification given by sending 
$(\cE_+,\cE_-)$ to 
\[
 0 \subset \varpi\cE_+/\varpi \cE^0 \subset \cE_-/\varpi \cE^0 \subset \cE^0/\varpi \cE^0. 
\]

Assume that $(\cE_+, \cE_-)$ is a point of $Y_i$. 
The condition 
$\cE \subset \oF (\cE^{\vee})$ is equivalent to that 
the image of 
$\sW$ under the natural morphism 
\begin{equation}\label{eq:phi1def}
 \phi_1 \colon \cE_+/\cE_- \to \oF (\cE_-^{\vee})/\oF (\cE_+^{\vee})=\oF (\cE_-^{\vee}/\cE_+^{\vee}) 
\end{equation}
is contained in $\oF (\sW^{\perp})$. 
The condition 
$\oF (\cE^{\vee}) \subset \frac{1}{\varpi} \cE$ 
is equivalent to that 
the image of 
$\oF (\varpi \sW^{\perp})$ under the natural morphism 
\[
 \phi_2 \colon \cE^0 /\varpi \cE^0 \to \cE_+/\cE_- 
\]
is contained in $\sW$. 
We put 
\[
 X_i = p_0^{-1} (X_{\mu^*,\nu_i^*}(\tau_i^*)) \cap 
 p_1^{-1}(Y_i). 
\]
Then $X_i$ is the subscheme of 
$p_1^{-1}(Y_i)$ cut out by the conditions
\begin{align}
\phi_1 (\sW) \subset \oF (\sW^{\perp}), \label{eq:cdphi1} \\ 
\phi_2 (\oF (\varpi \sW^{\perp})) \subset \sW . \label{eq:cdphi2}
\end{align}

Let $p_0'$ and $p_1'$ be the restrictions of 
$p_0$ and $p_1$ to $X_i$ respectively. 
We have 
\[
 \xymatrix{
 p_0^{-1} (X_{\mu^*,\nu_i^*}(\tau_i^*))   \ar[d]_-{p_0} &  
 X_i \ar[l] \ar[r] \ar[dr]_-{p_1'} 
 \ar[ld]^-{p_0'}  & 
  p_1^{-1}(Y_i) \ar[d]^-{p_1} \\ 
 X_{\mu^*,\nu_i^*}(\tau_i^*) 
 & & Y_i  .
 }
\]
We note that $p_0$ and $p_0'$ are 
isomorphisms over 
$\mathring{X}_{\mu^*}^{\bfb_i,x_0}(\tau_i^*)$.

\begin{lem}\label{lem:ImZi}
The inverse image 
$p_0^{-1}(\mathring{X}_{\mu^*}^{\bfb_i,x_0}(\tau_i^*))$ is contained in $X_i$. 
\end{lem}
\begin{proof}
Let $(\cE,\cE_+,\cE_-)$ be a point of 
$p_0^{-1}(\mathring{X}_{\mu^*}^{\bfb_i,x_0}(\tau_i^*))$. 
Then we have $\cE_-=\cE \cap \cE^0$. 
By the condition 
$\oF (\cE^{\vee}) \subset \frac{1}{\varpi} \cE$, 
we have 
\[
 \varpi \oF (\cE_-^{\vee})= \varpi \oF ((\cE \cap \cE^0)^{\vee}) =\varpi (\oF (\cE^{\vee}) + \cE^0) \subset \cE . 
\]
Hence we have 
\[
 \varpi \oF (\cE_-^{\vee}) \subset \cE \cap \cE^0 =\cE_-. 
\]
This means that $(\cE_+,\cE_-)$ is a point of $Y_i$. 
\end{proof}

Let $\sG_{Y_i}$ denote the restriction of $\sG_i$ to $Y_i$. 
We have an isomorphism 
\begin{equation}\label{eq:isoGr}
 \Psi_{Y_i} \colon p_1^{-1}(Y_i) \simeq \Par_{t_i}(\sG_{Y_i})
\end{equation}
induced by $\Psi$.

\begin{thm}\label{thm:pZi}
The closed subscheme 
$X_i \subset p_1^{-1}(Y_i) \simeq \Par_{t_i}(\sG_{Y_i})$ 
is defined by the condition 
$\phi_1 (\sW) \subset \oF (\sW^{\perp})$. 
\end{thm}
\begin{proof}
It suffices to show that 
the condition \eqref{eq:cdphi2} 
is automatic. 
The condition \eqref{eq:cdphi2} is equivalent to 
$\varpi \oF (\cE^{\vee}) \subset \cE$ under \eqref{eq:isoGr}. 
Let $(\cE,\cE_+,\cE_-)$ be a point of 
$p_1^{-1}(Y_i)$. Then we have 
\[
 \varpi \oF^{-1}(\cE_-^{\vee}) \subset \cE_- 
 \subset \cE . 
\]
By taking the dual, we have 
$\varpi \oF (\cE^{\vee}) \subset \cE_-$. 
Hence 
the condition 
$\varpi \oF (\cE^{\vee}) \subset \cE$ 
is satisfied. 
\end{proof}

\begin{prop}\label{prop:Xopen}
The scheme 
$\mathring{X}_{\mu^*}^{\bfb_i,x_0}(\tau_i^*)$ 
is isomorphic to the subscheme of 
$p_1^{-1}(Y_i)$ 
defined by the condition 
$\cE \subset \oF (\cE^{\vee})$ and 
$\cE \cap \cE^0 =\cE_-$. 
\end{prop}
\begin{proof}
Let $\mathring{X}_i$ and $\mathring{X}_{\mu^*,\nu_i^*}(\tau_i^*)$ 
be the inverse images of $\mathring{\Gr}_{\nu_i^*}$ 
in $X_i$ and $X_{\mu^*,\nu_i^*}(\tau_i^*)$. 
By Theorem \ref{thm:pZi}, 
$\mathring{X}_i$ is equal to the subscheme of 
$p_1^{-1}(Y_i)$ 
defined by the condition 
$\cE \subset \oF (\cE^{\vee})$ and 
$\cE \cap \cE^0 =\cE_-$. 
The natural morphism 
$p_0'^{-1}(\mathring{X}_{\mu^*}^{\bfb_i,x_0}(\tau_i^*)) \to 
\mathring{X}_{\mu^*}^{\bfb_i,x_0}(\tau_i^*)$ 
is an isomorphism. Hence it suffices to show that 
$p_0'^{-1}(\mathring{X}_{\mu^*}^{\bfb_i,x_0}(\tau_i^*))=\mathring{X}_i$. 

By Lemma \ref{lem:ImZi}, we have 
$p_0'^{-1}(\mathring{X}_{\mu^*}^{\bfb_i,x_0}(\tau_i^*))\subset \mathring{X}_i$. 
On the other hand, 
$\mathring{X}_i$ is contained in 
$p_0^{-1}(\mathring{X}^{b_i,x_0}_{\mu^*}(\tau_i^*))$, 
since 
$\mathring{X}_{\mu^*,\nu_i^*}(\tau_i^*)
=\mathring{X}^{b_i,x_0}_{\mu^*}(\tau_i^*)$ 
by Lemma \ref{lem:Gracl}. 
\end{proof}

\section{Intersections}\label{sec:Int}
Let $x,x' \in J_{\tau}(F)/J_{\tau}(\cO_F)$. 
Let $\Lambda_x$ and $\Lambda_{x'}$ be the lattices of $V$ 
determined by $x$ and $x'$. 
We put 
\[
 l_{x,x'}=\length_{\cO_F}\left( \Lambda_x/(\Lambda_x \cap \Lambda_{x'}) \right) . 
\]
We note that 
\[
 l_{x,x'}=\length_{\cO_F}\left( (\Lambda_x +\Lambda_{x'})/\Lambda_x \right)=\length_{\cO_F}\left( \Lambda_{x'}/(\Lambda_x \cap \Lambda_{x'}) \right) 
\]
by taking dual with respect to the hermitian pairing. 
We assume that $\varpi^2 \Lambda_{x'} \subset \Lambda_x$. This assumption is satisfied if $X_{\mu^*}^{\bfb_i,x}(\tau_i^*) \cap X_{\mu^*}^{\bfb_{i'},x'}(\tau_{i'}^*)$ is non-empty for some $i$ and $i'$, because if $\cE$ is in the intersection we have $\varpi \cE_{x'} \subset \cE \subset \varpi^{-1} \cE_x$. 

Let $\cE_x$ and $\cE_{x'}$ be the modifications of 
$\cE^0$ corresponding to $\Lambda_x$ and $\Lambda_{x'}$. 
Let $\bfP_{x,x'}$ be the parabolic subgroup of $\bfG$ that is the stabilizer of the filtration 
\[
 \varpi \Lambda_x \subset 
 \varpi^2 \Lambda_{x'} + \varpi \Lambda_x 
 \subset 
 (\Lambda_x \cap \varpi \Lambda_{x'}) +\varpi \Lambda_x \subset 
 (\Lambda_x \cap \Lambda_{x'}) +\varpi \Lambda_x \subset \Lambda_x. 
\]
We note that 
$\varpi \Lambda_x \subset \Lambda_{x'}$ if and only if 
$\varpi \Lambda_{x'} \subset \Lambda_x$ 
by taking dual with respect to the hermitian pairing. 
We put 
\begin{align*}
 d_1 &=\dim ((\varpi^2 \Lambda_{x'} + \varpi \Lambda_x)/\varpi \Lambda_x), \\
 d_2 &=\dim (((\Lambda_x \cap \varpi \Lambda_{x'}) +\varpi \Lambda_x)/\varpi \Lambda_x). 
\end{align*}
We have 
\begin{equation}\label{eq:d12l}
 d_1 +d_2 =l_{x,x'} 
\end{equation} 
since 
\begin{align*}
  \Lambda_{x'}/( \Lambda_{x'} \cap \varpi^{-1} \Lambda_x) &\simeq \varpi \Lambda_{x'}/(\varpi \Lambda_{x'} \cap \Lambda_x) \simeq (\varpi \Lambda_{x'} + \Lambda_x)/ \Lambda_x \\ 
  &\simeq (\varpi^2 \Lambda_{x'} + \varpi \Lambda_x)/ \varpi\Lambda_x ,\\ 
  ( \Lambda_{x'} \cap \varpi^{-1} \Lambda_x)/( \Lambda_{x'} \cap \Lambda_x) &\simeq 
  ( \varpi \Lambda_{x'} \cap  \Lambda_x)/( \varpi \Lambda_{x'} \cap \varpi \Lambda_x) \\ &\simeq ((\Lambda_x \cap \varpi \Lambda_{x'}) +\varpi \Lambda_x)/\varpi \Lambda_x. 
\end{align*}
In the identification \eqref{eq:ttbiso} for $\bfb_i$, we use $\delta_{\lambda_i}=-\nu_i^*$ if $1 \leq i \leq r-1$ and 
$\delta_{\lambda_r}=-\varepsilon_{r+1}$ if $i=r$. 

\subsection{Intersection of components for $\nu_i$ and $\nu_{i'}$, where $i,i' \neq r$ if $n$ is even.}

Let $1 \leq i, i' \leq r$. Assume that $i\neq r$ and $i' \neq r$ if $n$ is even. 

\subsubsection{Different hyperspecial subgroups}
We assume that $x \neq x'$. 
For a subscheme $X$ of $X_{\mu^*}^{\bfb_i,x}(\tau_i^*)$, 
let 
$X_{\bfP_{x,x'},[w]}$ be the inverse image of 
$X_{I_{n}^{i-1,n-i}}([1])_{\bfP_{x,x'},[w]}^{\pf}$ 
under 
$X \hookrightarrow X_{\mu^*}^{\bfb_i,x}(\tau_i^*) \to X_{I_{n}^{i-1,n-i}}([1])^{\pf}$. 

We recall that 
\[
 \mathring{X}_{\mu^*}^{\bfb_{i'},x'}(\tau_{i'}^*)=
 X_{\mu^*,\nu_{i'}^*}^{\bfa_{i'},x'}(\tau_{i'}^*) \setminus X_{\mu^*}^{\bfb_{i'-1},x'}(\tau_{i'-1}^*). 
\]

Assume that $i \leq i'$. 
For $j_1,j_2 \in \bN$ such that 
$i-1-d_2 \leq j_1 \leq i-1$ and 
$d_2 -i \leq j_2 \leq n-i-d_2 -j_1$, 
we define $w_{j_1,j_2} \in S_n$ by 
\[
 w_{j_1,j_2} (j)= 
 \begin{cases}
 j+j_1 & \textrm{if $i-j_1 \leq j \leq d_2$,}\\ 
 j+i-j_1-d_2 -1 & \textrm{if $d_2 +1 \leq j \leq d_2 +j_1$,}\\ 
 j+j_2 & \textrm{if $n-j_2-i+1 \leq j \leq n-d_2$,}\\ 
 j+d_2 -i-j_2 & \textrm{if $n-d_2 +1 \leq j \leq n-d_2 +j_2$,}\\ 
 j & \textrm{otherwise.} 
 \end{cases}
\]

We put 
$\cE_{x,x'}^+=(\cE_x + \cE_{x'}) \cap \frac{1}{\varpi} \cE_x$ 
and 
$\cE_{x,x'}^-=(\cE_x \cap \cE_{x'})+\varpi \cE_x$. 
Since $X_{I_{n}^{i-1,n-i}}([1])_{\bfP_{x,x'},[w_{j_1,j_2}]}^{\pf}$ parametrizes two subspaces of $\Lambda_x /\varpi \Lambda_x$ satisfying some conditions, there are universal vector bundles $\cV_{+} \subset \cV_- \subset \cE_x /\varpi \cE_x$ on $X_{I_{n}^{i-1,n-i}}([1])_{\bfP_{x,x'},[w_{j_1,j_2}]}^{\pf}$. 
We put 
\begin{align*}
    \cE_+ = \frac{1}{\varpi} \pi_x^{-1} (\cV_+) \subset \frac{1}{\varpi} \cE_x, \quad 
    \cE_- = \pi_x^{-1} (\cV_-) \subset \cE_x 
\end{align*}
where $\pi_x \colon \cE_x \to \cE_x /\varpi \cE_x$ is the natural projection. 
Those are the same as the restrictions of $(\cE_+,\cE_-)$ in \S \ref{sebsec:nonmin} under the identification $Y_i=X_{I_{n}^{i-1,n-i}}([1])^{\mathrm{pf}}$ in that subsection. 

We note that 
\[
 \length ((\cE_+ +\cE_{x,x'}^+ )/ \cE_{x,x'}^+ )=j_1, \quad \length ((\cE_- +\cE_{x,x'}^- )/\cE_{x,x'}^- )=j_2. 
\]
We put $\cE_{+,-} =\cE_+ \cap (\cE_- +\cE_{x'})$ and $d_{j_1,j_2}=j_2 -j_1 +2i-1-d_2$. 
We note that 
\begin{equation}\label{eq:incE+-}
 \cE_{+,-}=\cE_- + \cE_+ \cap \cE_{x'} \subset 
 \oF (\cE_-^{\vee}) \cap (\oF (\cE_+^{\vee}) + \cE_{x'}) 
 =\oF (\cE_{+,-}^{\vee}) 
\end{equation}
using 
$\cE_+ \subset \oF (\cE_-^{\vee})$ 
and 
$\cE_- \subset \oF (\cE_+^{\vee})$. 

\begin{lem}\label{lem:lengE+-}
We have 
$\length (\cE_{+,-}/\cE_-)=d_{j_1,j_2}$. 
Further $\cE_{+,-}/\cE_-$ is a vector bundle on 
$X_{I_{n}^{i-1,n-i}}([1])_{\bfP_{x,x'},[w_{j_1,j_2}]}^{\pf}$. 
\end{lem}
\begin{proof}
We have 
\begin{align*}
 \length (\cE_+/\cE_{+,-})&= \length ((\cE_+ +\cE_{x'})/(\cE_- +\cE_{x'})) \\ 
 &=\length ((\cE_+ +\cE_{x'})/(\varpi \cE_x + \cE_{x'} )) - \length ((\cE_- +\cE_{x'})/(\varpi \cE_x + \cE_{x'} )) \\ 
  &=\length ((\cE_+ +\cE_{x'})/( \cE_x + \cE_{x'} )) + 
  \length ((\cE_x + \cE_{x'})/(\varpi \cE_x + \cE_{x'} )) 
  - j_2 \\ 
  &=j_1 + 
  \length (( \cE_x \cap \varpi \cE_{x'} )/(\varpi \cE_x \cap \varpi\cE_{x'})) 
  - j_2 =j_1+d_2 -j_2. 
\end{align*}
Hence we obtain the first claim. 
By the above equalities, 
$\length (\cE_+/\cE_{+,-})$ is constant on 
$X_{I_{n}^{i-1,n-i}}([1])_{\bfP_{x,x'},[w_{j_1,j_2}]}^{\pf}$. Hence $\cE_+/\cE_{+,-}$ is a vector bundle on 
$X_{I_{n}^{i-1,n-i}}([1])_{\bfP_{x,x'},[w_{j_1,j_2}]}^{\pf}$ 
by \cite[Lemma 7.3]{BhScProWG}. 
Therefore $\cE_{+,-}/\cE_-$ is also a vector bundle. 
\end{proof}

Let $\sG_{j_1,j_2}$ be the restriction of $\sG$ to 
$X_{I_{n}^{i-1,n-i}}([1])_{\bfP_{x,x'},[w_{j_1,j_2}]}^{\pf}$. 
Let
\[
 t_{j_1,j_2} \in \mathrm{Oc}(\mathrm{Dyn}(\sG_{j_1,j_2}))(X_{I_{n}^{i-1,n-i}}([1])_{\bfP_{x,x'},[w_{j_1,j_2}]}^{\pf})  
\] 
denote the restriction of $t_i$ in \eqref{eq:tidef}. 
Let $\sP_{j_1,j_2}$ be the parabolic subgroup of $\sG_{j_1,j_2}$ 
determined by 
\[
 \cE_- \subset \cE_+ \cap (\cE_- +\cE_{x'}) \subset \cE_+. 
\]

We put $l_{j_1,j_2}=i'-1-j_2 -d_1$. 
We define $s_{j_1,j_2} \in S_{2i-1}$ by 
\[
 s_{j_1,j_2}(j)=  
 \begin{cases}
 j+l_{j_1,j_2} 
 & \textrm{if $
 i -l_{j_1,j_2} 
 \leq j \leq d_{j_1,j_2}$,} \\ 
 j +i-1 -d_{j_1,j_2} -l_{j_1,j_2}
 & \textrm{if $d_{j_1,j_2} +1 \leq j \leq 
 d_{j_1,j_2}+l_{j_1,j_2} 
 $,} \\ 
 j & \textrm{otherwise.} 
 \end{cases}
\]
Let $r_{j_1,j_2}$ be the element of 
\[
 (q_{\sG_{j_1,j_2}}^{-1}(\bft (\sP_{j_1,j_2}),t_{j_1,j_2}))(X_{I_{n}^{i-1,n-i}}([1])_{\bfP_{x,x'},[w_{j_1,j_2}]}^{\pf})
\]
corresponding to $[s_{j_1,j_2}]$ by 
Lemma \ref{lem:igiso} \ref{en:Wphi}.

\begin{prop}\label{prop:inter}
Assume that 
$\mathring{X}_{\mu^*}^{\bfb_i,x}(\tau_i^*) \cap \mathring{X}_{\mu^*}^{\bfb_{i'},x'}(\tau_{i'}^*)$ 
is not empty. 
Then we have $l_{x,x'} \leq i+i'-1$. 

The subscheme 
$\mathring{X}_{\mu^*}^{\bfb_i,x}(\tau_i^*) \cap \mathring{X}_{\mu^*}^{\bfb_{i'},x'}(\tau_{i'}^*) \subset \mathring{X}_{\mu^*}^{\bfb_{i},x}(\tau_{i}^*)$ is the locus defined by the condition that 
$\cE_x +\varpi \cE_{x'} \subset \cE_+ \subset \frac{1}{\varpi} \cE_{x,x'}^-$, 
$\varpi \cE_{x,x'}^+ \subset \cE_- \subset \cE_x \cap \frac{1}{\varpi} \cE_{x'}$, 
\begin{equation}\label{eq:lengE-}
 \length ((\cE_- +\cE_{x,x'}^-)/ \cE_{x,x'}^- ) \leq  [(i'-i+d_2 -d_1 )/2] 
\end{equation} 
and 
\begin{equation}\label{eq:lengsum}
 \length ((\cE +\cE_{+,-} )/ \cE_{+,-}) + \length ((\cE_- +\cE_{x,x'}^-)/ \cE_{x,x'}^- ) = i'-1 -d_1. 
\end{equation} 
In particular,  
$\mathring{X}_{\mu^*}^{\bfb_i,x}(\tau_i^*) \cap \mathring{X}_{\mu^*}^{\bfb_{i'},x'}(\tau_{i'}^*)$ 
is the union of 
$\left( \mathring{X}_{\mu^*}^{\bfb_i,x}(\tau_i^*) \cap \mathring{X}_{\mu^*}^{\bfb_{i'},x'}(\tau_{i'}^*) \right)_{\bfP_{x,x'},[w_{j_1,j_2}]}$ 
for 
$j_1,j_2 \in \bN$ such that $j_1 +d_2 -i \leq j_2 \leq j_1 +d_2-i +1$, 
\begin{align*}
    &i-1-d_2 \leq j_1 \leq i-1-d_1
    ,\\  
 &
 i' -i -d_1 \leq 
 j_2 \leq \min \{ 
 [(i'-i+d_2 -d_1 )/2] 
 , n-i-d_2 -j_1 
 \}. 
\end{align*}
Further we have 
\[
 \left( \mathring{X}_{\mu^*}^{\bfb_i,x}(\tau_i^*) \cap \mathring{X}_{\mu^*}^{\bfb_{i'},x'}(\tau_{i'}^*)  \right)_{\bfP_{x,x'},[w_{j_1,j_2}]} = 
 \mathring{X}_{\mu^*}^{\bfb_i,x}(\tau_i^*)_{\bfP_{x,x'},[w_{j_1,j_2}]} 
 \cap \Par_{t_{j_1,j_2}} 
(\sG_{j_1,j_2}; \sP_{j_1,j_2})_{r_{j_1,j_2}}^{\pf} . 
\] 
\end{prop}
\begin{proof}
The intersection 
$\mathring{X}_{\mu^*}^{\bfb_i,x}(\tau_i^*) \cap
 \mathring{X}_{\mu^*}^{\bfb_{i'},x'}(\tau_{i'}^*)$ 
is parametrized by 
$\cE \dashrightarrow \cE_x$ which is equal to 
$\nu_i^*$ such that 
$\cE \subset \oF (\cE^{\vee}) \subset \frac{1}{\varpi} \cE$ 
and 
$\cE \dashrightarrow \cE_{x'}$ 
is equal to $\nu_{i'}^*$. 
Let $\cE$ be a point of 
$\mathring{X}_{\mu^*}^{\bfb_i,x}(\tau_i^*) \cap \mathring{X}_{\mu^*}^{\bfb_{i'},x'}(\tau_{i'}^*)$. 
We put 
\[
 \cE_+=\cE +\cE_x, \quad 
 \cE_-=\cE \cap \cE_x, \quad 
 \cE_+'=\cE +\cE_{x'}, \quad
 \cE_-'=\cE \cap \cE_{x'}. 
\]
Then we have 
\begin{align*}
&\length (\cE_x/\cE_-)=i, \quad 
\length (\cE/\cE_-)=i-1, \\ 
&\length (\cE_{x'}/\cE_-')= i', \quad 
\length (\cE/\cE_-') = i' -1. 
\end{align*}
Hence we have 
$\length(\cE_-/(\cE_- \cap \cE_-')) \leq i' -1$ and 
$\length (\cE_x/(\cE_- \cap \cE_-')) \leq i+i' -1$. 
Therefore the inclusion 
$\cE_- \cap \cE_-' \subset \cE_x \cap \cE_{x'}$ implies that 
\[
 l_{x,x'} \leq i+i' -1. 
\]

We have 
$\cE_x +\varpi \cE_{x'} \subset \cE_+$ and 
$\varpi \cE_{x,x'}^+ =\varpi \cE_x +(\cE_x \cap \varpi \cE_{x'}) \subset \cE_-$, 
since $\varpi \cE_{x'} \subset \cE$. 
We have 
$\varpi \cE_+ \subset \cE_x \cap (\cE_{x'}+\varpi \cE_x)=\cE_{x,x'}^-$ and  
$\cE_- \subset \cE_x \cap \frac{1}{\varpi} \cE_{x'}$, 
since $\varpi \cE_+ \subset \cE_x$ and 
$\varpi \cE \subset \cE_{x'}$. 

We put $j_1=\length ((\cE_+ +\cE_{x,x'}^+ )/ \cE_{x,x'}^+ )$ 
and $j_2=\length ((\cE_- +\cE_{x,x'}^- )/\cE_{x,x'}^- )$. 
We have 
\begin{align*}
 \length (\cE_- / (\cE_- \cap \cE_-')) 
 &=\length (\cE_- / (\cE_- \cap \cE_{x'})) 
 =j_2 +\length ((\cE_-\cap \cE_{x,x'}^-) / (\cE_- \cap \cE_{x'}))\\ 
 &=j_2 +\length (\cE_{x,x'}^-/ (\cE_x \cap \cE_{x'})) = j_2 +d_1. 
\end{align*}
We have 
\[
 j_2 +i-d_2 = \length ((\cE_- + \cE_{x,x'}^-)/\cE_-) 
 \leq 1+\length ((\oF (\cE_+^{\vee}) + \cE_{x,x'}^-)/\oF (\cE_+^{\vee}))
\] 
since $\length (\oF (\cE_+^{\vee})/\cE_-)=1$. 
Further we have 
\begin{align*}
 \length & ((\oF (\cE_+^{\vee}) + \cE_{x,x'}^-)/\oF (\cE_+^{\vee}))
 = \length (\cE_+ / (\cE_+ \cap \cE_{x,x'}^+)) 
 \leq \length (\cE_+ / (\cE_-' +\cE_x )) \\ 
 &\leq 
 \length (\cE / (\cE_- +\cE_-'))
 =  i'-1-\length (\cE_- / (\cE_- \cap \cE_-')) 
 =i'-1-j_2 -d_1. 
\end{align*}
Therefore we obtain 
$j_2  \leq  [(i'-i+d_2 -d_1 )/2]$. 

Further, 
$j_1 +d_2 -i \leq j_2 \leq j_1 +d_2-i +1$ 
follows from $\length (\oF (\cE_-^{\vee})/\cE_+)=1$. 
This implies 
$i-1-d_2 \leq j_1$ and 
$d_2 -i \leq j_2$. 
We have 
$j_1 \leq i-1-d_1$ and 
$j_2 \leq n-i-d_2 -j_1$ 
by the inclusions 
$\cE_x +\varpi \cE_{x'} \subset \cE_+ \cap \cE_{x,x'}^+$
and 
$\cE_+ +\cE_{x,x'}^+ \subset \cE_- \cap \cE_{x,x'}^-$. 
The equality 
\[
 \length ((\cE +\cE_{+,-} )/ \cE_{+,-}) + \length ((\cE_- +\cE_{x,x'}^-)/ \cE_{x,x'}^- ) = i'-1 -d_1 
\] 
and $\length ((\cE +\cE_{+,-} )/ \cE_{+,-}) \leq i-1$ imply 
$j_2 \geq i'-i-d_1$. 

We have 
\begin{align*}
 \length & ((\cE +\cE_{+,-} )/ \cE_{+,-}) 
 = \length (\cE / (\cE \cap \cE_{+,-}))
 = \length (\cE / (\cE \cap (\cE_- + \cE_{x'} ))) \\ 
 &= \length ((\cE +\cE_{x'})/ (\cE_- + \cE_{x'} )) 
 = \length ((\cE +\cE_{x'})/ \cE_{x'} ) 
 -\length ((\cE_- + \cE_{x'} )/ \cE_{x'} ) \\ 
 &= \length ((\cE +\cE_{x'})/ \cE_{x'} ) 
 -\length (\cE_-/ (\cE_- \cap \cE_{x'} ))  . 
\end{align*}
Hence, 
$\length  ((\cE +\cE_{+,-} )/ \cE_{+,-}) = i'-1 - j_2 -d_1$ 
if and only if 
$\length ((\cE +\cE_{x'})/ \cE_{x'} ) = i'-1$. 
This implies the last claim. 
\end{proof}

\subsubsection{Same hyperspecial subgroup}
Assume that $x=x'$. 
It suffices to consider the case where $x=x'=x_0$, 
since all the hyperspecial subgroups are conjugate. 

Let $2 \leq i \leq [(n-1)/2]$. 
Let $(\cE,\cE_+,\cE_-)$ be a point of 
$X_i$. 
Let $s$ be the rank of 
$(\cE \cap \cE^0)/\cE_-$. 
We put $\cV_1=\cE/\cE_-$ 
and take 
$\cV_2 \subset \cE^0/\cE_-$ and 
$\cV_3 \subset \cE_+/\cE_-$ 
such that projections induce 
isomorphisms 
$\cV_2 \simeq (\cE+\cE^0)/\cE$ 
and 
$\cV_3 \simeq \cE_+/(\cE+\cE^0)$. 
An open neighbourhood of 
$(\cE,\cE_+,\cE_-)$ in $\Gr(i-1,\sV_{Y_i})$ 
under \eqref{eq:isoGr} 
is given by 
$\Hom (\cV_1,\cV_2 \oplus \cV_3)$ 
sending 
$f \in \Hom (\cV_1,\cV_2 \oplus \cV_3)$ 
to the inverse image $\cE_f$ of 
\[
 \{ v+f(v) \mid v \in \cV_1 \} \subset \cE_+/\cE_- 
\]
in $\cE_+$. 
By Theorem \ref{thm:pZi}, the condition that 
$\cE_f$ belongs to $X_i$ 
is equivalent to 
\begin{equation}\label{eq:condEf}
 \langle v+f(v) , \oF (v'+f(v')) \rangle =0 
\end{equation}
in $\varpi^{-1}W_{\cO_F}(R) / W_{\cO_F}(R)$ 
for $v,v' \in \cV_1$. 
We write 
$f$ as $f_2 +f_3$ for 
$f_2 \in \Hom (\cV_1,\cV_2)$ 
and 
$f_3 \in \Hom (\cV_1,\cV_3)$ 
For $v,v' \in (\cE \cap \cE^0)/\cE_-$, 
the condition \eqref{eq:condEf} 
is equivalent to 
\begin{equation}\label{eq:cond0Ef}
 \langle v +f_2(v), \oF (f_3(v')) \rangle + 
 \langle f_3(v) , \oF (v'+f_2(v')+f_3(v')) \rangle 
 =0 
\end{equation}
in $\varpi^{-1}W_{\cO_F}(R) / W_{\cO_F}(R)$. 

Take a basis $v_1,\ldots,v_{i-1}$ of $\cV_1$ 
such that $v_1,\ldots ,v_s$ form a basis of $(\cE \cap \cE^0)/\cE_-$. 
Take a basis $v_i,\ldots,v_{2i-s-1}$ of $\cV_2$ and 
a basis $v_{2i-s},\ldots ,v_{2i-1}$ of $\cV_3$. 
Write $f(v_l)$ as $x_{l,i} v_i + \cdots + x_{l,2i-1} v_{2i-1}$. 
Then the condition \eqref{eq:cond0Ef} is equivalent to 
\begin{equation*}
 \langle v_l +\sum_{j=i}^{2i-s-1} x_{l,j} v_j, \oF (\sum_{k=2i-s}^{2i-1} x_{m,k}v_k ) \rangle + 
 \langle \sum_{k=2i-s}^{2i-1} x_{l,k} v_k, \oF (v_m+\sum_{j=i}^{2i-1} x_{m,j} v_j) \rangle 
 =0 
\end{equation*}
for $1 \leq l,m \leq s$. 
We can write this as 
\[
 (\langle v_l +\sum_{j=i}^{2i-s-1} x_{l,j} v_j, \oF ( v_k ) \rangle)_{l,k} 
 (x_{m,k}^q)_{k,m} =
 -( x_{l,k})_{l,k} ( \langle v_k, \oF (v_m+\sum_{j=i}^{2i-1} x_{m,j} v_j) \rangle )_{k,m}. 
\]
Taking the determinant, we obtain 
\begin{align*}
 \det (x_{l,k})_{l,k} \Bigl( 
 & 
 \det (\langle v_l +\sum_{j=i}^{2i-s-1} x_{l,j} v_j, \oF ( v_k ) \rangle)_{l,k} 
 (\det (x_{l,k})_{l,k})^{q-1} \\
 &-(-1)^s 
 \det ( \langle v_k, \oF (v_m+\sum_{j=i}^{2i-1} x_{m,j} v_j) \rangle )_{k,m} 
 \Bigr) =0. 
\end{align*}
The condition $\cE_f \cap \cE^0=\cE_-$ is equivalent 
to $\det (x_{l,k})_{l,k} \neq 0$. 
Hence, 
if $(\cE,\cE_+,\cE_-)$ belongs to 
the closure of $p_0'^{-1}(\mathring{X}_{\mu^*}^{\bfb_i,x_0}(\tau_i^*))$, 
then we have 
$\det ( \langle v_k, \oF (v_m ) \rangle )_{k,m} =0$. 
This means 
$\oF^{-1}(\cE_+^{\vee}) \subset \cE$. 
Hence we have obtained the following proposition: 

\begin{prop}
The intersection 
\[
 p_0'^{-1}(X_{\mu^*}^{\bfb_{i-s},x_0}(\tau_{i-s}^*)) \cap 
\ol{p_0'^{-1}(\mathring{X}_{\mu^*}^{\bfb_i,x_0}(\tau_i^*))} 
\]
is contained in the locus defined by 
the condition $\oF^{-1}(\cE_+^{\vee}) \subset \cE$. 
\end{prop}

Conversely, we assume that 
$\oF^{-1}(\cE_+^{\vee}) \subset \cE$. 
Then we may assume that 
$v_1$ is a basis of $\oF^{-1}(\cE_+^{\vee})/\cE_-$, 
$v_i$ is an element of $(\oF (\cE^{\vee}) \cap \cE^0)/\cE_-$ 
lifting a basis of 
$(\oF (\cE^{\vee}) \cap \cE^0)/(\cE \cap \cE^0)$ 
such that $v_i \notin \oF^{-1}(\cE_+^{\vee})/\cE_-$ 
and 
$v_{2i-s}$ is an element of 
$\oF (\cE^{\vee})/\cE_-$ lifting a basis of 
$(\oF (\cE^{\vee}) + \cE^0)/(\cE + \cE^0)$. 
Further, we may assume that 
$v_2,\ldots,v_{i-1}$ and 
$v_{2i-1}, \ldots , v_{2i-s+1}, v_{2i-s-1},\ldots , v_{i+1}$ 
form dual base with respect to the pairing 
\[
 \cE/(\oF^{-1}(\cE_+^{\vee})) \times \cE_+/(\oF (\cE^{\vee})); (v,v') \mapsto 
 \langle \oF (v),v' \rangle . 
\] 
and that $\langle v_j,\oF (v_k) \rangle=0$ 
for $i+1 \leq j \leq 2i-s-1$ and $i \leq k \leq 2i-1$. 
Then the condition \eqref{eq:condEf} is equivalent to 
\begin{align}
 &\langle \sum_{j=2i-s}^{2i-1} x_{l,j} v_j, \oF (\sum_{k'=i}^{2i-1} x_{m,k'} v_{k'}) \rangle + 
 \begin{cases}
 \langle v_l +x_{l,i} v_i, \oF (\sum_{k=2i-s}^{2i-1} x_{m,k}v_k ) \rangle & \textrm{if $1 \leq l \leq r$,} \\
 \langle v_l +x_{l,i} v_i, \oF (\sum_{k'=i}^{2i-1} x_{m,k'}v_{k'} ) \rangle & \textrm{if $s+1 \leq l \leq i-1$}
 \end{cases} \notag \\
 &+
 \begin{cases}
 0 & \textrm{if $m=1$,}\\
 x_{l,2i+1-m} & \textrm{if $2 \leq m \leq s$,}\\
 x_{l,2i-m} & \textrm{if $s+1 \leq m \leq i-1$} 
 \end{cases} =0 \label{eq:relx} 
\end{align}
for $1 \leq l,m \leq i-1$. 

We put 
\[
 y=\det (x_{l,j})_{1 \leq l \leq s,\, 2i-s \leq j \leq 2i-1}. 
\]
We want to show that the quotient of 
$k[[x_{l,j}]]_{1 \leq l \leq i-1,\, i \leq j \leq 2i-1}$ 
by the relation \eqref{eq:relx} is nonzero after inverting $y$.

\begin{prop}\label{prop:inti2}
\begin{enumerate}
\item\label{en:i2clint}
The intersection 
\[
 p_0'^{-1}(X_{\mu^*}^{\bfb_{1},x_0}(\tau_{1}^*)) \cap 
\ol{p_0'^{-1}(\mathring{X}_{\mu^*}^{\bfb_2,x_0}(\tau_2^*))} 
\]
is equal to the locus defined by 
the condition $\oF^{-1}(\cE_+^{\vee}) = \cE$. 
\item\label{en:12int}
We have an isomorphism 
$X_{\mu^*}^{\bfb_1,x_0}(\tau_1^*) \cap X_{\mu^*}^{\bfb_2,x_0}(\tau_2^*) \simeq X_{I_{n}^1}^{\oF,\oF^3}([1],[1])^{\mathrm{pf}}$ 
given by 
$\cE \mapsto \cE^{\vee}/ \cE^0$. Further, this intersection is irreducible. 
\end{enumerate}
\end{prop}
\begin{proof}
In this case, 
\eqref{eq:relx} becomes 
\[
 \langle x_{1,3} v_3, \oF (x_{1,2} v_2 +x_{1,3} v_3) \rangle + 
 \langle v_1 +x_{1,2} v_2, \oF ( x_{1,3} v_3 ) 
 \rangle =0. 
\]
If the quotient of 
$k[[x_{1,2},x_{1,3}]]$ 
by this relation is zero after inverting $x_{1,3}$, 
there is a positive integer $N$ such that 
$x_{1,3}^N$ is divisible by 
\[
 \langle x_{1,3} v_3, \oF (x_{1,2} v_2 +x_{1,3} v_3) \rangle + 
 \langle v_1 +x_{1,2} v_2, \oF ( x_{1,3} v_3 ) 
 \rangle 
\]
in $k[[x_{1,2},x_{1,3}]]$. 
This does not happen because 
$\langle v_3, \oF (v_2 ) \rangle \neq 0$, which follows from $v_2 \notin \oF^{-1}(\cE_+^{\vee})/\cE_-$. 
Hence we have \ref{en:i2clint}. 
The claim \ref{en:12int} follows from 
\ref{en:i2clint} and Lemma \ref{lem:irred}, since $X_{I_{n}^1}^{\oF,\oF^2,\oF^3}([1],\leq [s_1],[1])=X_{I_{n}^1}^{\oF,\oF^3}([1],[1])$. 
\end{proof}

By Proposition \ref{prop:inti2}, 
$X_{\mu^*}^{\bfb_1,x_0}(\tau_1^*) \cap X_{\mu^*}^{\bfb_2,x_0}(\tau_2^*)$ 
is isomorphic to the perfect closed subscheme of $(\bP^{n-1})^{\mathrm{pf}}$ defined by two equations 
\[
 \sum_{i=1}^n x_i x_{n+1-i}^q=0, \quad 
 \sum_{i=1}^n x_i x_{n+1-i}^{q^3}=0. 
\]
Since all non-degenerate hermitian forms on $\bF_{q^2}^n$ 
are isomorphic, the above scheme is 
isomorphic to 
the perfect closed subscheme of $(\bP^{n-1})^{\mathrm{pf}}$ defined by two equations 
\[
 \sum_{i=1}^n x_i^{q+1}=0, \quad 
 \sum_{i=1}^n x_i^{q^3+1}=0. 
\]

\subsection{Intersection of components for $\nu_i$ and $\nu_r$ when $n$ is even.} 

Assume that $n$ is even in this subsection. 
We put $g_r=\varpi^{\delta_{\lambda_r} + \nu_r^*}=\varpi^{-(\varepsilon_{r+1} + \cdots + \varepsilon_n )}$, 
$\Lambda_{x,r}=g_r \Lambda_x$ and $\cE_{x,r}=g_r \cE_x$. 

\begin{prop}\label{prop:intnurxx'}
Let $i \neq r$. Assume that  
$\mathring{X}_{\mu^*}^{\bfb_i,x}(\tau_i^*) \cap X_{\mu^*}^{\bfb_r,x'}(\tau_r^*)$ 
is non-empty. 
Then we have 
$\Lambda_{x'} \subset \Lambda_{x,r}$,   
$\varpi \Lambda_{x,r} \subset \varpi^{-1} \Lambda_{x'}$ and 
\begin{equation}\label{eq:ExrEx'+r}
\length (\Lambda_{x,r}/\Lambda_{x,r} \cap \Lambda_{x'})=\length (\Lambda_{x'}/\Lambda_{x,r} \cap \Lambda_{x'}) +r.
\end{equation} 
Let $\bfP_{x,r,x'}$ be the parabolic subgroup of $\bfG$ that is the stabilizer of the filtration 
\[
 \cE_{x'} \subset \varpi \cE_{x,r} +\cE_{x'} \subset 
 \cE_{x,r} \cap \varpi^{-1} \cE_{x'} \subset \varpi^{-1} \cE_{x'}. 
\]
We put $j_1=\length ((\varpi \Lambda_{x,r} +\Lambda_{x'})/\Lambda_{x'})$, $j_2=\length ((\Lambda_{x,r} \cap \varpi^{-1} \Lambda_{x'})/\Lambda_{x'})$ and 
define $w_r \in S_n$ by 
\[
 w_r (j)=
 \begin{cases}
j+i-1 & \textrm{if $j_1 +1 \leq j \leq j_2$,}\\
j-j_2 -i +r & \textrm{if $j_2+1 \leq j \leq j_2+i -1$,}\\
j & \textrm{otherwise.}
 \end{cases}
\]
Then we have 
\[
 \mathring{X}_{\mu^*}^{\bfb_i,x}(\tau_i^*) \cap X_{\mu^*}^{\bfb_r,x'}(\tau_r^*) = X_{\mu^*}^{\bfb_r,x'}(\tau_r^*)_{\bfP_{x,r,x'},[w_r]}. 
\]
\end{prop}
\begin{proof}
By Proposition \ref{prop:nur}, 
$X_{\mu^*}^{\bfb_r,x'}(\tau_r^*)$ is parametrized by $\cE \dashrightarrow \cE_{x'}$ 
bounded by $\nu_r^*$ such that $\varpi \cE \subset \oF (\cE^{\vee})$. 
By the identification \eqref{eq:ttbiso}, the subscheme 
$\mathring{X}_{\mu^*}^{\bfb_i,x}(\tau_i^*) \cap X_{\mu^*}^{\bfb_r,x'}(\tau_r^*) \subset X_{\mu^*}^{\bfb_r,x'}(\tau_r^*)$ is given by the conditions 
\begin{equation}\label{eq:EExr}
\length ((\cE +\cE_{x,r})/\cE_{x,r})=i-1, \quad 
\length (\cE_{x,r}/\cE \cap \cE_{x,r})=i 
\end{equation}
and $\varpi \cE_{x,r} \subset \cE \subset \varpi^{-1} \cE_{x,r}$. 
Let $\cE \in \mathring{X}_{\mu^*}^{\bfb_i,x}(\tau_i^*) \cap X_{\mu^*}^{\bfb_r,x'}(\tau_r^*)$. 
Since $\varpi \cE_{x,r} \subset \cE$, we have 
$\oF (\cE^{\vee}) \subset \cE_{x,r}$. Hence $\cE_{x'} \subset \cE \subset \oF (\cE^{\vee}) \subset \cE_{x,r}$. 
We also have $\varpi \cE_{x,r} \subset \cE \subset \varpi^{-1} \cE_{x'}$. 
By the equality 
\begin{align*}
\length & ((\cE+\cE_{x,r})/\cE_{x,r})+\length (\cE_{x,r}/\cE_{x,r} \cap \cE_{x'}) \\ 
 &= 
\length ((\cE+\cE_{x,r})/\cE) +\length (\cE /\cE_{x'}) + \length (\cE_{x'}/\cE_{x,r} \cap \cE_{x'}) , 
\end{align*}
$\length (\cE /\cE_{x'})=r-1$ and \eqref{eq:EExr}, we have 
\eqref{eq:ExrEx'+r}. 

Since we have \eqref{eq:ExrEx'+r}, by the above argument, for any $\cE$ parametrizing $X_{\mu^*}^{\bfb_r,x'}(\tau_r^*)$ the condition 
\eqref{eq:EExr} holds if and only if $\length ((\cE +\cE_{x,r})/\cE_{x,r})=i-1$, 
which is equivalent to 
$\length ((\cE +(\cE_{x,r}\cap \varpi^{-1} \cE_{x'}))/(\cE_{x,r}\cap \varpi^{-1} \cE_{x'}))=i-1$. 
Therefore the subscheme 
$\mathring{X}_{\mu^*}^{\bfb_i,x}(\tau_i^*) \cap X_{\mu^*}^{\bfb_r,x'}(\tau_r^*) \subset X_{\mu^*}^{\bfb_r,x'}(\tau_r^*)$ is given by the conditions 
$\length ((\cE +(\cE_{x,r}\cap \varpi^{-1} \cE_{x'}))/(\cE_{x,r}\cap \varpi^{-1} \cE_{x'}))=i-1$ and $\varpi \cE_{x,r} +\cE_{x'} \subset \cE$. This implies the claim. 
\end{proof}

Assume that $x \neq x'$.

\begin{prop}\label{prop:intneqrr}
Assume that $X_{\mu^*}^{\bfb_r,x}(\tau_r^*) \cap X_{\mu^*}^{\bfb_r,x'}(\tau_r^*)$ is not empty. 
Then we have $l_{x,x'} \leq r-1$ and 
$\varpi \Lambda_x \subset \Lambda_{x'}$. 
The intersection $X_{\mu^*}^{\bfb_r,x}(\tau_r^*) \cap X_{\mu^*}^{\bfb_r,x'}(\tau_r^*)$ is parametrized by $\cE \dashrightarrow \cE_x$ bounded by 
$\nu_r^*$ such that 
$\varpi \cE \subset \oF (\cE^{\vee})$ and $\cE \dashrightarrow \cE_{x'}$ is also bounded by 
$\nu_r^*$. In particular, it is isomorphic to
$$\left\{ H \in \mathrm{Gr}^{\pf}(r-1-l_{x,x'}, \varpi^{-1} (\Lambda_x \cap \Lambda_{x'}) / (\Lambda_x + \Lambda_{x'}) ) \mid H \subset \mathrm{Frob}(H^\perp) \right\}.$$
\end{prop}

\begin{proof}

Assume that $\cE$ is a point of $X_{\mu^*}^{\bfb_r,x}(\tau_r^*) \cap X_{\mu^*}^{\bfb_r,x'}(\tau_r^*)$. Since $\varpi \cE \subset \oF (\cE^{\vee})$ and both $\cE \dashrightarrow \cE_x$ and $\cE \dashrightarrow \cE_{x'}$ are bounded by $\nu_r^*$, we have the following chain conditions:
\begin{align*}
 &\cE_x \subset \cE \subset  \varpi^{-1} \oF (\cE^{\vee}) \subset \varpi^{-1} \cE_x , \\ 
  & \cE_{x'} \subset \cE \subset  \varpi^{-1} \oF (\cE^{\vee}) \subset \varpi^{-1} \cE_{x'}.   
\end{align*}
The inclusion follows from 
$\cE_x \subset \cE \subset \varpi^{-1} \cE_{x'}$. 
Note that $\length ( \varpi^{-1} \oF (\cE^{\vee})  / \cE  )=2$, while both $\length (\cE  / \cE_x )$ and $\length ( \cE  / \cE_{x'} )$ are $r-1$.
Then $\cE_x \cap \cE_{x'}$ and $\cE$ are related by
\[ 
 \cE_x + \cE_{x'} \subset \cE \subset  \varpi^{-1} \oF (\cE)  \subset \varpi^{-1} (\cE_x \cap \cE_{x'}) . 
\]
Since $l_{x,x'} = \length ( (\cE_x + \cE_{x'})/ \cE_x )$, 
we have
$$l_{x,x'} = r -1 - \length(\cE/(\cE_x + \cE_{x'} )).$$
Hence we have $l_{x,x'} \leq r-1$. 

The isomorphism in the claim is given by sending $\cE$ to $\cE/ (\cE_x + \cE_{x'}) \subset  \varpi^{-1} (\cE_x \cap \cE_{x'}) /  (\cE_x + \cE_{x'})$. 
\end{proof}

\section{Example}\label{sec:Exa}
In this section, we study in details the case where $n=6$. 
We identify the moduli parametrizing 
modification $\cE \subset \cE_x$ bounded by $\nu_1^*$ 
with $(\bP^5)^{\mathrm{pf}}$ 
by taking a basis of $\Lambda_x$ such that 
the Hermitian pairing is the standard one. 
Let $\bP_{x,x',+}$ be the projective subspace of 
$(\bP^5)^{\mathrm{pf}}$ defined by the condition 
$\varpi \cE_{x,x'}^+ \subset \cE$. 
Let $\bP_{x,x',-}$ be the projective subspace of 
$(\bP^5)^{\mathrm{pf}}$ defined by the condition 
$\cE_{x,x'}^- \subset \cE$. 
We note that 
$\bP_{x,x',+}$ and $\bP_{x,x',-}$ are 
isomorphic to 
$(\bP^{5-d_2})^{\mathrm{pf}}$ 
and 
$(\bP^{d_2 -1})^{\mathrm{pf}}$ respectively. 

In the following, we freely use Proposition \ref{prop:inter} to determine the range of $j_1$ and $j_2$. 

\subsection{Intersection of components for $\nu_1$}
We may assume that $x \neq x'$. 
The intersection is not empty only if 
$l_{x,x'}=1$, since if $\cE$ is in the intersection we have $\cE \subset \cE_x \cap \cE_{x'}$ and $\cE \subset \cE_x$ is bounded by $\nu_1^*$. 
In this case, $d_1=0$, $d_2=1$, $j_1=j_2=0$ by \eqref{eq:d12l} and $d_1 \leq d_2$. 
The intersection is $\bP_{x,x',-}$, which is a point given by 
$\cE_x \cap \cE_{x'}$. 

\subsection{Intersection of components for $\nu_1$ and $\nu_2$}
If $x=x'$, then the intersection is 
isomorphic to the perfect closed subscheme of $(\bP^5)^{\mathrm{pf}}$ defined by two equations 
\[
 \sum_{i=1}^6 x_i x_{7-i}^q=0, \quad 
 \sum_{i=1}^6 x_i x_{7-i}^{q^3}=0 
\]
by Proposition \ref{prop:inti2}. 

We assume that $x\neq x'$. 
We can check claims in \S \ref{sssec:d01} and \S \ref{sssec:d02} using Proposition \ref{prop:inter}. Especially, the conditions \eqref{eq:lengE-} and \eqref{eq:lengsum} are automatically satisfied in these cases. 

\subsubsection{$d_1=0$, $d_2=1$}\label{sssec:d01}
In this case, $j_1=0$ and $j_2=1$. 
The intersection is equal to 
the perfect closed subscheme of $\bP_{x,x',+}$ defined by equation 
\[
 \sum_{i=1}^5 x_i x_{6-i}^q=0.  
\]

\subsubsection{$d_1=0$, $d_2=2$}\label{sssec:d02}
In this case, $j_1=0$ and $j_2=1$. 
The intersection is 
$\bP_{x,x',-}$, which is isomorphic to 
$(\bP^1)^{\mathrm{pf}}$. 

\begin{rem}
If $d_1=d_2=1$, then there is no $j_2 \in \bN$ satisfying the condition in Proposition 
\ref{prop:inter}. 
\end{rem}

\subsection{Intersection of components for $\nu_2$}
Let 
$(\cE_+,\cE_-)$ be a point of 
$X_{I_{6}^{1,4}}([1])^{\pf}$. 
The hermitian pairing on $V$ induces a pairing on 
$\cE_+/\cE_-$ 
since we have $\cE_+ \subset \oF (\cE_-^{\vee})$ 
and 
$\cE_- \subset \oF (\cE_+^{\vee})$. 
We take 
a basis $v_1, v_2, v_3$ of 
$\cE_+/\cE_-$ such that 
$v_1 \in \oF^{-1}(\cE_+^{\vee})/\cE_-$, 
$v_2 \in \cE_x/\cE_-$. 
Let $\cE$ be a point of 
$\mathring{X}_{\mu^*}^{\bfb_2,x}(\tau_2^*)$ 
in the fiber of $(\cE_+,\cE_-)$ under 
\[
 \pi \colon 
 \mathring{X}_{\mu^*}^{\bfb_2,x}(\tau_2^*)
 \to X_{I_{6}^{1,4}}([1])^{\pf} . 
\]
We can take a generator  
$v=x_1 v_1 +x_2 v_2 +v_3$ 
of $\cE/\cE_-$ 
for $x_1,x_2 \in k$, since 
$\cE \not\subset \cE_x$. 
Then we have 
\begin{equation*}
  \langle v, \oF (v)\rangle = 
  x_1 
 \langle v_1,\oF (v_3) \rangle + 
 x_2  
 \langle v_2,\oF (v_3) \rangle +
 x_2^q 
 \langle v_3,\oF (v_2) \rangle +
  \langle v_3,\oF (v_3) \rangle 
\end{equation*}
because 
$\langle w,\oF (w') \rangle =0$ 
for $w,w' \in \cE_x/\cE_-$ and 
$\langle v_3,\oF (v_1) \rangle =0$. 
Hence the fiber of $(\cE_+,\cE_-)$ under 
$\pi$ is defined by 
\begin{equation}\label{eq:Efibdefcond}
  x_1 
 \langle v_1,\oF (v_3) \rangle + 
 x_2  
 \langle v_2,\oF (v_3) \rangle +
 x_2^q 
 \langle v_3,\oF (v_2) \rangle +
  \langle v_3,\oF (v_3) \rangle =0 
\end{equation}
in $(\bA^2)^{\pf}$. 
We note that 
$(\langle v_1,\oF (v_3) \rangle,   
 \langle v_2,\oF (v_3) \rangle) \neq (0,0)$ 
 because $v_3 \notin \cE_x/\cE_-$. 

We describe the fiber of 
\[
 \pi_{j_1,j_2} \colon 
 \mathring{X}_{\mu^*}^{\bfb_2,x}(\tau_2^*)_{\bfP_{x,x'},[w_{j_1,j_2}]} 
 \cap \Par_{t_{j_1,j_2}} 
(\sG_{j_1,j_2}; \sP_{j_1,j_2})_{r_{j_1,j_2}}^{\pf} 
 \to X_{I_{6}^{1,4}}([1])_{\bfP_{x,x'},[w_{j_1,j_2}]}^{\pf} 
\]
and determine its dimension 
when 
\[
  \left( \mathring{X}_{\mu^*}^{\bfb_2,x}(\tau_2^*) \cap \mathring{X}_{\mu^*}^{\bfb_{2},x'}(\tau_{2}^*)  \right)_{\bfP_{x,x'},[w_{j_1,j_2}]}  
\]
is not empty. 
We recall that $\length (\cE_{+,-}/\cE_-)=d_{j_1,j_2}$ by Lemma 
\ref{lem:lengE+-}. 
We note that if $\cE \subset \cE_{+,-}$ the condition 
\eqref{eq:Efibdefcond} is automatic by 
\eqref{eq:incE+-}. 
In the following $5$ cases, the condition on the relation between $\cE$ and $\cE_{+,-}$ follows from  \eqref{eq:lengsum}. 

\subsubsection{$d_1=0$, $d_2=1$}
In this case, $0 \leq j_1 \leq 1$ and $j_2=0$. 
We have $d_{j_1,j_2}=2-j_1$. 
The fiber of $\pi_{j_1,0}$ is given by the condition 
$\cE \not\subset \cE_{+,-}$, where
the dimension of the fiber is $1$. 

\subsubsection{$d_1=0$, $d_2=2$}
In this case, $0 \leq j_1 \leq 1$ and $0 \leq j_2 \leq 1$. 
We have $d_{j_1,j_2}=1-j_1+j_2$. 
The fiber of $\pi_{j_1,0}$ is given by the condition 
$\cE \not\subset \cE_{+,-}$, where
the dimension of the fiber is $1$. 
The fiber of $\pi_{j_1,1}$ is given by the condition 
$\cE \subset \cE_{+,-}$, where 
the dimension of the fiber is $1-j_1$ since 
$\length (\cE_{+,-}/\cE_-)=2-j_1$. 

\subsubsection{$d_1=1$, $d_2=1$} 
In this case, $j_1=0$ and $j_2=0$. 
We have $d_{j_1,j_2}=2$. 
The fiber of $\pi_{0,0}$ is given by the condition 
$\cE \subset \cE_{+,-}$, where 
the dimension of the fiber is $1$ since 
$\length (\cE_{+,-}/\cE_-)=2$. 

\subsubsection{$d_1=0$, $d_2=3$}
In this case, $j_1 =0$ and $j_2 = 1$. 
We have $d_{j_1,j_2}=1$. 
The fiber of $\pi_{0,1}$ is given by the condition 
$\cE = \cE_{+,-}$, where 
the dimension of the fiber is $0$. 

\subsubsection{$d_1=1$, $d_2=2$}
In this case, $j_1=0$ and $j_2=0$. 
We have $d_{j_1,j_2}=1$. 
The fiber of $\pi_{0,0}$ is given by the condition 
$\cE = \cE_{+,-}$, where 
the dimension of the fiber is $0$. 

\subsection{Intersection of components for $\nu_i$ ($1 \leq i \leq 2$) and $\nu_3$}
In this case, the intersection is given by $X_{\mu^*}^{\bfb_3,x'}(\tau_3^*)_{\bfP_{x,3,x'},[w_3]}$ as Proposition \ref{prop:intnurxx'}, and 
$X_{\mu^*}^{\bfb_3,x'}(\tau_3^*)$ is isomorphic to $X_{I_{6}^{2}}([1])^{\mathrm{pf}}$ by Proposition \ref{prop:nur}. 

\subsection{Intersection of components for $\nu_3$}

In the following two cases, the claims follow from Proposition \ref{prop:intneqrr}. 

\subsubsection{$l_{x,x'}=1$}
The intersection is isomorphic to 
the perfection of the Fermat hypersurface 
defined by 
\[
 x_1 x_4^q + x_2 x_3^q + x_3 x_2^q +x_4 x_1^q =0 
\]
in $\bP^3$. 

\subsubsection{$l_{x,x'}=2$}
The intersection is a point given by $\cE_x +\cE_{x'}$. 

\section{Shimura variety}\label{sec:ShVar}
In this section, we explain how the study of $\mathring{X}_{\mu^*}^{\bfb_i,x_0}(\tau_i^*)$ is related to the supersingular locus of a Shimura variety, recalling previously known results. 

Let $E$ be a quadratic imaginary field, and let $\sfV$ be an $n$-dimensional Hermitian space over $E$ with signature $(2,n-2)$ at infinity. Fix a prime $p \neq 2$ inert in $E$. 
Further assume that $\sfV \otimes_{E} \bQ_{p^2}$ contains a self-dual $\bZ_{p^2}$ lattice $\Lambda$. Let $\sfG = \GU (\sfV)$ be the general associated unitary group. 
We put $G=\GU (\Lambda)$ as before. 

We take a basis of 
$\sfV_{\bC} = \sfV \otimes_{E} \bC$ 
over $\bC$ 
such that the Hermitian form is given by the 
matrix 
$\diag (1_2,-1_{n-2})$. 
Let $h \colon \Res_{\bC/\bR} {\Gm}_{\bC} \to \sfG_{\bR}$ be the morphism of algebraic groups over $\bR$ such that 
$h(z)$ corresponds to $\diag (z \cdot 1_2,\bar{z} \cdot 1_{n-2})$ for $z \in \bC^{\times}$ 
under 
\[
 \sfG (\bR) \subset \Aut_{\bC} (\sfV_{\bC}) \simeq 
 \GL_n(\bC) , 
\]
where the last isomorphism is given by 
the basis taken above. 
Let $\sfX$ be the $\sfG (\bR)$-conjugacy class of 
$h$. 
Then $(\sfG,\sfX)$ is a Shimura datum. 

We have an isomorphism 
\[
 (\Res_{\bC/\bR} {\Gm}_{\bC})_{\bC} 
 \simeq {\Gm}_{\bC} \times {\Gm}_{\bC} 
\]
of algebraic groups over $\bC$ induced by the isomorphism 
$\bC \otimes_{\bR} \bC \simeq \bC \times \bC;\ a \otimes b \mapsto (ab,\bar{a}b)$. 
We define $\mu_h$ by the composition 
\[
 {\Gm}_{\bC} \hookrightarrow 
 {\Gm}_{\bC} \times {\Gm}_{\bC} 
 \simeq (\Res_{\bC/\bR} {\Gm}_{\bC})_{\bC} \xra{h_{\bC}} 
 \sfG_{\bC} , 
\]
where the first morphism is the inclusion into the first factor. 
Let $\mu \colon {\Gm}_E \to \sfG_E$ 
be the morphism of algebraic 
over $E$ such that 
$\mu(z)$ corresponds to 
$(\diag (z \cdot 1_2, 1_{n-2}),z)$ 
for $z \in E^{\times}$ 
under the isomorphism 
\[
 \sfG_E \simeq \GL_n (E) \times {\Gm}_E 
\]
given by taking a basis of $\sfV$ over $E$. 
Then $\mu_h$ and $\mu_{\bC}$ 
are in the same $\sfG (\bC)$-conjugacy class. 
We note that the reflex field 
$E(\sfG,\sfX)$ of $(\sfG,\sfX)$ is 
$E$ if $n \neq 4$ and $\bQ$ if $n=4$. 

Let $K^p \subset \sfG (\Af^p )$ be a sufficiently small open compact subgroup. 
Let $K_p \subset \sfG (\bQ_p )$ be a hyperspecial subgroup. 
We put $K=K^p K_p \subset \sfG (\Af )$. 
Let $\Sh_K (\sfG,\sfX)$ be the canonical model over 
$E(\sfG,\sfX)$ of the Shimura variety attached to 
$(\sfG,\sfX)$ and $K$. 
Let $\sS_K (\sfG,\sfX)$ be the canonical integral model of 
$\Sh_K (\sfG,\sfX)$ over 
$\cO_{E(\sfG,\sfX),(p)}$ constructed in 
\cite{KisIntShab}. 

Let $\bfS_{K}(\sfG,\sfX)$ be the perfection of 
$\sS_K (\sfG,\sfX) \otimes \ol{\bF}_p$. 
We have the Newton map 
\[
 \cN \colon \bfS_{K}(\sfG,\sfX) (\ol{\bF}_p) \to B(G,\mu^* )  
\]
as in \cite[7.2.7]{XiZhCycSh}. 
Let $[b] \in B(G,\mu^* )$ be the basic element. 
We write $\bfS_{K}(\sfG,\sfX)_{[b]}$ for the 
closed perfect subscheme of $\bfS_{K}(\sfG,\sfX)$ 
defined by $\cN^{-1}([b])$. 
We call $\bfS_{K}(\sfG,\sfX)_{[b]}$ the supersingular locus of 
$\bfS_{K}(\sfG,\sfX)$. 

\begin{rem}
In \cite{KotPointSh}, a moduli space of abelian schemes with additional structures is constructed. 
It is isomorphic to 
a finite union of integral models of Shimura varieties. 
Under the isomorphism, a point of 
$\bfS_{K}(\sfG,\sfX)_{[b]}$ corresponds to 
a supersingular abelian variety. 
\end{rem}

We take a point $x \in \bfS_{K}(\sfG,\sfX)_{[b]}(\ol{\bF}_p)$. 
We put $L=W(\ol{\bF}_p)[\frac{1}{p}]$. 
Then we have a basic element $b_x \in G(L)$ 
and an algebraic group $I_x$ over $\bQ$ 
as in 
\cite[7.2.9]{XiZhCycSh}. 
We have embeddings 
$I_x(\bQ) \subset \sfG (\Af^p)$ and  
$I_x(\bQ) \subset J_{b_x} (\bQ_p)$ 
as in \cite[7.2.13]{XiZhCycSh}. 
Then we have the isomorphism 
\begin{equation}\label{eq:unifss}
 I_x(\bQ) \backslash X_{\mu^*} (b_x) \times \sfG (\Af^p) 
 / K^p \xrightarrow{\sim} \bfS_{K}(\sfG,\sfX)_{[b]} 
\end{equation}
by \cite[Corollary 7.2.16]{XiZhCycSh}.
We use notations in \S \ref{sec:Irr} for $F=\bQ_p$.

\begin{prop}
We have $\dim \bfS_{K}(\sfG,\sfX)_{[b]}=n-2$. 
The irreducible components of 
$\bfS_{K}(\sfG,\sfX)_{[b]}$ are parametrized by 
\[
 \coprod_{1 \leq i \leq r} I_x(\bQ) \backslash (G(\bQ_p)/G(\bZ_p)) \times \sfG (\Af^p) 
 / K^p . 
\]
For sufficiently small $K^p$, a non-empty open subscheme of each irreducible component of $\bfS_{K}(\sfG,\sfX)_{[b]}$ is isomorphic to 
a non-empty open subscheme of $\mathring{X}_{\mu^*}^{\bfb_i,x_0}(\tau_i^*)$ 
for some $i$, which is described in \S \ref{sec:Irr}. 
\end{prop}
\begin{proof}
The first two claims follow from 
Proposition \ref{prop:irralluni} and \eqref{eq:unifss}. 
The last claim is proved in the same way as \cite[Theorem 6.1]{VolGUI}. 
\end{proof}

\noindent
Maria Fox\\
Department of Mathematics, 
Oklahoma State University, 
Stillwater, OK 74078, USA\\
maria.fox@okstate.edu\\

\noindent
Naoki Imai\\
Graduate School of Mathematical Sciences, The University of Tokyo, 
3-8-1 Komaba, Meguro-ku, Tokyo, 153-8914, Japan \\
naoki@ms.u-tokyo.ac.jp 

\end{document}